\documentclass[12pt]{amsart}
\usepackage{amsmath, amsfonts, amssymb, url}
 \usepackage[all]{xy}

\def\e{\epsilon}

\newcommand{\mB}{\mathcal{B}}

\newcommand{\mE}{\mathcal{E}}
\newcommand{\mF}{\mathcal{F}}
\newcommand{\mG}{\mathcal{G}}

\newcommand{\mM}{\mathcal{M}}
\newcommand{\mN}{\mathcal{N}}

\newcommand{\mP}{\mathcal{P}}

\newcommand{\mS}{\mathcal{S}}

\newcommand{\mW}{\mathcal{W}}

\newcommand{\fm}{\mathfrak{m}}
\newcommand{\fM}{\mathfrak{M}}

\newcommand{\fp}{\mathfrak{p}}

\newcommand{\fq}{\mathfrak{q}}

\newcommand{\fX}{\mathfrak{X}}

\newcommand{\bfC}{\mathbf{C}}

\newcommand{\bfF}{\mathbf{F}}

\newcommand{\bfQ}{\mathbf{Q}}

\newcommand{\bfT}{\mathbf{T}}

\newcommand{\bfZ}{\mathbf{Z}}

\newcommand{\Oo}{\mathcal{O}}

\newcommand{\OF}{\mathcal{O}_F}

\newcommand{\OL}{\mathcal{O}_L}

\newcommand{\ov}{\overline}

\newcommand{\be}{\begin{equation}}
\newcommand{\ee}{\end{equation}}
\newcommand{\bes}{\begin{equation*}}
\newcommand{\ees}{\end{equation*}}

\newcommand{\bs}{\begin{split}}
\newcommand{\es}{\end{split}}
\newcommand{\bss}{\begin{split*}}
\newcommand{\ess}{\end{split*}}

\newcommand{\bmat}{\left[ \begin{matrix}}
\newcommand{\emat}{\end{matrix} \right]}
\newcommand{\bsmat}{\left[ \begin{smallmatrix}}
\newcommand{\esmat}{\end{smallmatrix} \right]}

\newcommand{\bml}{\begin{multline}}
\newcommand{\eml}{\end{multline}}
\newcommand{\bmls}{\begin{multline*}}
\newcommand{\emls}{\end{multline*}}

\DeclareMathOperator{\Cl}{Cl}

\DeclareMathOperator{\Frob}{Frob}
\DeclareMathOperator{\Gal}{Gal}
\DeclareMathOperator{\GL}{GL}

\DeclareMathOperator{\Hom}{Hom}

\DeclareMathOperator{\Spec}{Spec}

\DeclareMathOperator{\val}{val}

\newcommand{\R}{\textup{Res}_{K/\mathbf{Q}}\hspace{2pt}}

\newcommand{\hs}{\hspace{2pt}}

\newcommand{\tr}{\textup{tr}\hspace{2pt}}

\usepackage[all]{xy}
\usepackage{amsthm}
\usepackage{amssymb, amsfonts}
\SelectTips{cm}{10}\UseTips
\theoremstyle{plain}
\newtheorem{thm}{Theorem}
\newtheorem{prop}[thm]{Proposition}
\newtheorem{cor}[thm]{Corollary}
\newtheorem{lemma}[thm]{Lemma}

\theoremstyle{definition}
\newtheorem{definition}[thm]{Definition}

\newtheorem{rem}[thm]{Remark}

\numberwithin{thm}{section}
\numberwithin{equation}{section}

\usepackage{tikz}
\usetikzlibrary{arrows}
\usetikzlibrary{decorations.markings}

\usepackage{graphicx}
\usepackage{epsf}
\usepackage{verbatim} 
\usepackage{mathrsfs}
\usepackage{amsthm}
\author{Tobias Berger and Krzysztof Klosin}
\title{$R=T$ theorems for weight one modular forms}
\begin{document}
 
\thanks{The first author's research was supported by the EPSRC Grant EP/R006563/1. The second author was supported by a Collaboration for Mathematicians Grant \#578231  from the Simons Foundation
 and by a PSC-CUNY award jointly funded by the Professional Staff Congress and the City
University of New York.}

\begin{abstract} We prove modularity of certain residually reducible ordinary 2-dimensional $p$-adic Galois representations with determinant a finite order odd character $\chi$. For certain non-quadratic $\chi$ we prove an $R=T$ result for $T$ the weight 1 specialisation of the Hida Hecke algebra acting on non-classical weight 1 forms, under the assumption that  no two Hida families congruent to an Eisenstein series cross in weight 1. For quadratic $\chi$ we prove that the quotient of $R$ corresponding to deformations split at $p$ is isomorphic to the Hecke algebra acting on classical CM weight 1 modular forms. \end{abstract}
\maketitle

\section{Introduction}
 Let $\rho: G_{\bfQ} \to \GL_2(\ov{\bfQ}_p)$ be a Galois representation unramified outside a finite set $\Sigma$ of primes with $p \in \Sigma$ which is residually reducible, $p$-distinguished and ordinary at $p$. Suppose that $\det \rho = \chi \epsilon^{k-1}$ where $\chi$ is a finite order character, $\epsilon$ is the $p$-adic cyclotomic character and $k$ is a positive integer (such that $\chi\epsilon^{k-1}$ is odd) and that  the associated residual representation has semi-simplification $1\oplus\ov{\chi} \ov{\epsilon}^{k-1}$. If $k\geq 2$, the modularity of such representations by modular forms of weight $k$ was proved by Skinner and Wiles \cite{SkinnerWiles97, SkinnerWiles99} (recently generalised by Lue Pan \cite{Pan21} to the non-ordinary case). The case of $k=1$ is different because while one can still expect that the Galois representations should arise from weight one modular forms, in general not all such forms are classical, i.e., there are purely $p$-adic weight one ordinary modular forms. This phenomenon was first observed by Mazur and Wiles \cite{MazurWilesCompositio}.

 In this article we prove the first modularity theorem for residually reducible Galois representations with $k=1$ where the Galois representations in question are modular but not necessarily by a classical modular  form of weight one.  In fact, it was shown by Dummigan and Spencer \cite{DummiganSpencer} that if $\chi$ is not quadratic there are no classical modular forms of weight 1 whose associated residual Galois representation has semi-simplification $1 \oplus \ov{\chi}$ (see Remarks \ref{CM1} and  \ref{rem5.5}). In that case we prove a modularity theorem by purely $p$-adic weight 1 modular forms. If $\chi$ is quadratic we prove a modularity theorem by classical weight 1 modular forms with complex multiplication.

While we follow a well-established approach of identifying an appropriate deformation ring with a Hecke algebra, if the character is not quadratic we introduce some novel elements into the method which considerably shift the focus of the approach to dealing with some new challenges. In particular, we need to work with limits of modular Galois representations which, as opposed to the representations attached to forms of weight greater than or equal to 2,  are not automatically irreducible. The Hecke algebra in question is not a Hecke algebra acting on the space of classical weight one modular forms, but rather a certain quotient of the $\Lambda$-adic Hecke algebra which considerably complicates getting an ``$R\to T$''-map. The proof of the principality of the ideal of reducibility is also new. The key input on the automorphic side is Wiles' result from \cite{Wiles90} on $\Lambda$-adic Eisenstein congruences. 

On the other hand, if $\chi$ is quadratic we directly establish sufficiently many Eisenstein congruences of classical cusp forms of weight 1. To ensure modularity by classical forms we are also forced to work with a stronger deformation condition at $p$ (see Corollary \ref{lastcorollary}).

Let us now explain the contents of this paper in more detail.
Let $E$ be a finite extension of $\bfQ_p$ with integer ring $\Oo$, uniformizer $\varpi$ and residue field $\bfF$.  Let $G_{\Sigma}$ be the Galois group of the maximal Galois extension of $\bfQ$ unramified outside $\Sigma$. Let $\chi: G_{\Sigma} \to \Oo^{\times}$ be an odd Galois character associated with a Dirichlet character mod $Np$ whose $N$-part is primitive and let $\ov{\chi}: G_{\Sigma} \to \bfF^{\times}$ be its mod $\varpi$ reduction. We assume that  $\ov{\chi}|_{D_p}\neq 1$. In particular, we allow $\ov{\chi}$ to be unramified at $p$. Let $$\rho_0: G_{\Sigma} \to \GL_2(\bfF), \quad \rho_0=\bmat 1&*\\ 0& \ov{\chi}\emat\not\cong 1\oplus \ov{\chi}$$ be a continuous homomorphism. We study deformations $\rho$ of $\rho_0$ which are ordinary at $p$ (for precise definition see section \ref{deformationproblem}), and such that $\rho|_{I_{\ell}}=1\oplus \chi$ for all $\ell \in \Sigma$ with $\ell \equiv 1 $ mod $p$. We furthermore require that $\det \rho =\chi$, i.e., that $k=1$. Such a deformation problem is representable by a universal deformation ring $R$.  We then also study the deformation problem with the (stronger) assumption that $\rho|_{D_p}$ is split with corresponding universal deformation ring $R^{\rm split}$. We refer to these two cases as  ``ordinary'' and ``split''. 

We do not use the Taylor-Wiles method. Instead, we prove that there is a surjection from $R$ to a suitable Hecke algebra $T$, then show that reducible deformations are modular by demonstrating that $R/I \cong T/J$ for $I$ the reducibility ideal and $J$ the Eisenstein ideal. After establishing the principality of $I$ we then use the commutative algebra criterion from \cite{BergerKlosin11} to deduce $R=T$. For $R^{\rm split}$ we use a similar approach. 

However, to implement this general strategy we use very different routes in the ordinary and the split case. This dichotomy reflects the fact that in the first case we need to deal with non-classical, while in the latter one with  classical forms.

Let us first discuss the ordinary case.
In that case we work with Wiles' $\Lambda$-adic Hecke algebra $\bfT$ and consider a certain quotient $\bfT_1$ of it - its specialisation at weight one. More precisely, we take a localisation $\bfT_{1,\fm}$ of $\bfT_1$ at a maximal ideal corresponding to forms congruent to a weight 1 specialisation of a certain $\Lambda$-adic Eisenstein series $\mE$. 

To construct a map from $R$ to $\bfT_{1,\fm}$ we need to show that for each weight one specialisation $\mF_1$ of a Hida family $\mF$ congruent to $\mE$ there exists a lattice inside the Galois representation $\rho_{\mF_1}$ associated with $\mF_1$ which reduces to $\rho_0$. We do this, as is usually the case, by utilizing a result known as Ribet's Lemma, but this in turn requires irreducibility of $\rho_{\mF_1}$. For classical forms, this is a theorem of Deligne-Serre, but our $\mF_1$ need not be classical. As $\mF_1$ is a $p$-adic limit of classical forms in higher weights, and these have irreducible Galois representations,  the question becomes that  of proving that this irreducibility is preserved in the limit.  In general, of course, a limit of irreducible representations may be reducible. Here however, we show that in this context this does not happen, which is one of the key technical results of the paper (Theorem \ref{irreducibility}).

This  gives us a map from $R$ to $\prod_{\mF} \Oo$ where the product runs over all $\Lambda$-adic newforms congruent to $\mE$. A standard approach is then to identify the Hecke algebra with a subalgebra of such a product. Here however we encounter another problem, as $\bfT_{1,\fm}$ is  just an abstract quotient of the $\Lambda$-adic Hecke algebra and does not in general inject into $\prod_{\mF} \Oo$. The obstruction to this injection occurs when two Hida families congruent to $\mE$ cross at weight one. Let us note here that under our assumptions this cannot happen if the weight one specialisations are classical by a result of Bella{\"\i}che-Dimitrov (indeed, in their terminology we are in the regular case and cannot have real multiplication, so they prove that the eigencurve is \'etale at the corresponding point).

As the second key result of this article (Proposition \ref{inject3}), we prove that the lack of such crossings (i.e., the \'etaleness of the eigencurve at weight one) is indeed also a sufficient condition for having the desired injection in the general (also non-classical) case. In section \ref{Examples} we discuss examples when the non-crossing condition is satisfied.

Once we have a (surjective) map from $R$ to $\bfT_{1,\fm}$, we prove that it descends to an isomorphism $R/I\to \bfT_{1,\fm}/J_1$ roughly following the method of \cite{BergerKlosin13}, which boils down to bounding the orders of $R/I$ and $\bfT_{1,\fm}/J_1$. Here $I$ is the reducility ideal and $J_1$ is 
the (weight one specialisation of the) Eisenstein ideal. To bound $\bfT_{1.\fm}/J_1$ from below we use a theorem of Wiles from his proof of the Main Conjecture in \cite{Wiles90} which gives such a bound on the Eisenstein quotient of the $\Lambda$-adic Hecke algebra. For the corresponding upper bound on $R/I$ let us only mention that  we need to assume the cyclicity of the $\chi^{-1}$-eigenspace of $\Cl(\bfQ(\chi))$. A similar condition has been applied in various situations by Skinner-Wiles \cite{SkinnerWiles97},  the authors \cite{BergerKlosin13, BergerKlosin20}, and by Wake--Wang--Erickson \cite{WWE20}. For a full list of assumptions see section \ref{The residual representation}. 

To conclude that $R\cong \bfT_{1,\fm}$ we utilize the commutative algebra criterion from \cite{BergerKlosin11}, but to apply it we need to show that $I$ is a principal ideal. This is another major technical result of the paper that is needed in both the ordinary and the split case and uses different conditions in the two cases (Theorem \ref{principality}). In fact, the condition needed in the ordinary case excludes quadratic characters (see Remark \ref{exclusions}), so that case concerns exclusively non-classical forms. Hence for us the ordinary and the split case are in fact disjoint.

To complete our treatment of the modularity of residually reducible Galois representations with determinant a finite order character $\chi$ we prove in Theorem \ref{CMresult} that $R^{\rm split}=(\bfT_1^{\rm class})_\fm$ when $\chi$ is the character corresponding to an imaginary quadratic extension $F/\bfQ$ and $p$ is inert in $F/\bfQ$ and divides the class number of $F$. Here $(\bfT_1^{\rm class})_\fm$ is the localisation at the Eisenstein maximal ideal of the Hecke algebra acting on weight 1 classical cusp forms of level $d_F$ with complex multiplication. Whilst the usual methods for proving Eisenstein congruences do not apply in this case, it turns out that there is a very direct link here between elements of the Selmer group bounding $R^{\rm split}/I^{\rm split}$ and cusp forms congruent to the corresponding weight 1 Eisenstein series. To establish the required lower bound on the congruence module $(\bfT_1^{\rm class})_\fm/J$ we can therefore count the total depth of Eisenstein congruences provided by CM forms and apply the result of \cite{BergerKlosinKramer14}. This, in turn, requires us to know the principality of $J$ which we deduce from the principality of $I^{\rm split}$.

If $f$ is a classical weight one modular form, then by Deligne-Serre its Galois representation has finite image. However, there is no a priori reason why this should be true of an arbitrary deformation of $\rho_0$. In particular, if $f$ is classical and ordinary then $\rho_f|_{D_p}$ must be split (as it must be of finite order). Conjecturally this happens only for classical weight 1 forms. 
Under our assumptions we note that for $\chi$ unramified at $p$ our result $R^{\rm split}=(\bfT_1^{\rm class})_\fm$ (Theorem \ref{CMresult}) establishes the following equivalence (see Corollary \ref{lastcorollary}): an ordinary deformation $\rho:G_\Sigma \to \GL_2(\Oo)$ of $\rho_0$ is modular by a classical weight 1 form if and only if it is unramified at $p$ and $\chi$ is quadratic.

The modularity direction of this result is the analogue of that of Buzzard-Taylor \cite{BuzzardTaylor99} on the modularity of residually irreducible $p$-distinguished representations of $G_\Sigma$ that are unramified at $p$ (and therefore establishes another case of conjecture 5a in Fontaine-Mazur \cite{FontaineMazur95} that $p$-adic representations of $G_\Sigma$ that are unramified at $p$ have finite image).

Theorem \ref{CMresult} also complements the work of Castella, Wang-Erickson and Hida \cite{CastellaWangErickson21} in the residually irreducible case on Greenberg's conjecture that $\rho_f$ is split at $p$ if and only if $f$ is CM.

Let us note that the problem considered in \cite{SkinnerWiles97} is a related one and while \cite{SkinnerWiles97} assume throughout their article that $k\geq 2$, it appears to us that it would possible to infer a certain modularity result in the weight one case from their isomorphism of the ordinary universal deformation ring with the ordinary Hecke algebra. Even given that, however, there are significant differences in the setup and method. The residual representation considered in \cite{SkinnerWiles97} has the form $$\bmat \chi & * \\ 0 & 1\emat,$$ so has the opposite order of the characters on the diagonal to our $\rho_0$. Using this, the authors construct a certain reducible characteristic zero deformation of it, but its uniqueness (necessary for their method) requires that $\chi$ be ramified at $p$. In our setup such a deformation does not exist, and consequently our approach is different (not only because \cite{SkinnerWiles97} work in weight 2 and we work directly in weight 1). In fact, while both \cite{SkinnerWiles97} and we  assume that the $\chi$-part of the class group of the splitting field $F$ of $\chi$ is trivial, we work with the ideal of reducibility of $\rho_0$ and this requires another cyclicity assumption, namely that the $\chi^{-1}$ part of the class group of $F$ has dimension at most 1. However, as a benefit we do not need to assume that $\chi$  is ramified at $p$. 
Such unramified characters can occur in this context as we demonstrate in section \ref{Examples}.

\subsection{Acknowledgements}
We would like to thank Chris Skinner for teaching us about Wiles' proof of the Main Conjecture in Michigan in 2002. We are also grateful to Adel Betina for helpful comments, in particular regarding section \ref{Examples}. Finally, we would like to thank Neil Dummigan and Joe Kramer-Miller for enlightening conversations related to the topics of this article.

\section{Selmer groups} 

Let $p$ be a prime. 
For $\Sigma$ a finite set of finite places of $\bfQ$ containing $p$ we write $G_{\Sigma}$ for the Galois group of the maximal extension of $\bfQ$ unramified outside of $\Sigma$ and infinity. For any prime $\ell$ we write $D_{_{\ell}}\subset G_\Sigma$ for  a  decomposition group at $\ell$
and $I_{\ell} \subset D_\ell$ for the inertia subgroup.

We fix an embedding $\ov{\bfQ}_p \hookrightarrow \bfC$. Let $E$ be a finite extension of $\bfQ_p$. Write $\Oo$ for the valuation ring of $E$, $\varpi$ for a choice of a uniformizer and $\bfF$ for the residue field.

 For $\psi: G_\Sigma \to \Oo^\times$ a  non-trivial
character of order prime to $p$ 
we consider the $p$-adic coefficients $M=E(\psi)$, $E/\Oo(\psi)$ or $(\Oo/\varpi^n)(\psi)$ for $n \geq 1$. We also write $\ov{\psi}: G_{\Sigma} \to \bfF^{\times}$ for the mod $\varpi$ reduction of $\psi$. \begin{rem} \label{invariants1}  Note that if $G$ is a subgroup of $G_{\Sigma}$ such that $\psi|_G\neq 1$, then $(E/\Oo)(\psi)^G=0$. Indeed, as the order of $\psi$ is prime to $p$ the image of $\psi$ is contained in the prime-to-$p$ roots of unity of $\Oo$ and so $\psi$ is the Teichmüller lift of $\ov{\psi}$. This guarantees that if $\psi|_G\neq 1$ then there exists $\sigma \in G$ such that $\psi(\sigma) \not \equiv 1$ mod $\varpi$. \end{rem}

Let $\Sigma' \subset \Sigma$. For  $M$ as above we define the Selmer group $H^1_{\Sigma'}(\bfQ, M)$ to be the subgroup of $H^1(G_{\Sigma}, M)$ 
$$H^1_{\Sigma'}(\bfQ, M)=\ker(H^1(G_{\Sigma}, M) \to  \prod_{\ell \in \Sigma \backslash \Sigma'} (H^1(\bfQ_\ell, M)/H^1_{\rm f}(\bfQ_\ell, M))),$$ where the local conditions are defined as follows:

For $M=E(\psi)$ we take for all primes $\ell$, including $p$,
$$H^1_{\rm f}(\bfQ_\ell, M)=H^1_{\rm ur}(\bfQ_\ell, M)={\rm ker}(H^1(\bfQ_{\ell},M) \to H^1(\bfQ_{\ell, \rm ur},M)),$$  where $\bfQ_{\ell, \rm ur}$ is the maximal unramified extension of  $\bfQ_{\ell}$.
This induces conditions for $M=(E/\Oo)(\psi)$ and $(\Oo/\varpi^n)(\psi)$ via
$$H^1_{\rm f}(\bfQ_\ell, E/\Oo(\psi))={\rm im}(H^1_{\rm ur}(\bfQ_\ell, E(\psi)) \to H^1(\bfQ_{\ell}, E/\Oo(\psi)))$$  and
$$H^1_{\rm f}(\bfQ_\ell, (\Oo/\varpi^n)(\psi))=i_n^{-1} H^1_{\rm f}(\bfQ_\ell, E/\Oo(\psi)) \text{ for } i_n: H^1(\bfQ_\ell, (\Oo/\varpi^n)(\psi)) \to H^1(\bfQ_\ell, E/\Oo(\psi))$$ the natural map induced by the canonical injection $(\Oo/\varpi^n)(\psi)\to E/\Oo(\psi)$. 

For $\ell \neq p$ \cite{Rubin00} Lemma 1.3.5(iii) tells us that $H^1_{\rm f}(\bfQ_\ell, E/\Oo(\psi)))=H^1_{\rm ur}(\bfQ_\ell, E/\Oo(\psi)))$ since $(E/\Oo)(\psi)^{I_\ell}$ is divisible 
as $\psi$ has order prime to $p$. Indeed, if $\psi$ is unramified then the invariants are isomorphic to $E/\Oo$ as $\Oo$-modules,  hence divisible. If $\psi$ is ramified then the invariants are zero by Remark \ref{invariants1}.
By  the same \cite{Rubin00} Lemma 1.3.5(iii) $H^1_{\rm f}(\bfQ_\ell, \Oo(\psi)))$ (defined as preimage of $H^1_{\rm f}(\bfQ_\ell, E(\psi))$) agrees with ${\rm im}(H^1_{\rm ur}(\bfQ_{\ell}, \Oo(\psi))$, which by the proof of \cite{Rubin00} Lemma 1.3.8 also gives $H^1_{\rm ur}(\bfQ_\ell, (\Oo/\varpi^n)(\psi)) = H^1_{\rm f}(\bfQ_\ell, (\Oo/\varpi^n)(\psi))$.

For $\ell=p$ we also have $H^1_{\rm f}(\bfQ_p, E/\Oo(\psi))=H^1_{\rm ur}(\bfQ_p, E/\Oo(\psi)),$ by the proof of \cite{Rubin00} Proposition 1.6.2 as the order of $\psi$ is coprime to $p$. 
In addition an easy diagram chase like in the proof of \cite{Rubin00} Lemma 1.3.5 for $H^1_{\rm f}(K,T)$ shows that \be \label{urf} H^1_{\rm ur}(\bfQ_p, (\Oo/\varpi^n)(\psi)) \subset H^1_{\rm f}(\bfQ_p, (\Oo/\varpi^n)(\psi)).\ee

By \cite{Rubin00} Lemma 1.5.4 and Lemma 1.2.2(i) we have \begin{equation} \label{functoriality} H^1_{\Sigma'}(\bfQ, (\Oo/\varpi^n)(\psi))=H^1_{\Sigma'}(\bfQ, (E/\Oo)(\psi))[\varpi^n]\end{equation} since $(E/\Oo)(\psi)^{G_\Sigma}=0$ by Remark \ref{invariants1}.

\begin{prop}[\cite{Rubin00} Proposition 1.6.2] \label{clgroup}
$$H^1_{\emptyset}(\bfQ, E/\Oo(\psi)) \cong \Hom({\rm Cl}(\bfQ(\psi)), E/\Oo(\psi))^{{\rm Gal}(\bfQ(\psi)/\bfQ)}$$
\end{prop}

\begin{lemma} \label{lower bound 21} Let $\tilde\xi$ be a Dirichlet character.  Let $\xi: G_{\Sigma} \to \Oo^{\times}$ be the associated Galois character and write $\ov{\xi}$ for its mod $\varpi$ reduction.  Let $s$ be a positive integer. Set $W=E/\Oo(\xi^{-1})$ and $W_s=W[\varpi^s]$. Suppose $\ell \in \Sigma - \{p\}$ and let $\Sigma' \subset \Sigma$ with $\ell \not\in \Sigma'$. Assume that either \begin{itemize}
\item[(i)] $\xi_s:=\xi$ mod $\varpi^s$ is unramified at $\ell$ and $\ell\ov{\xi}(\Frob_{\ell})\neq 1$; 

or 

\item[(ii)] $\ov{\xi}$ is ramified at $\ell$. \end{itemize} Then one has $$H^1_{\Sigma' \cup \{\ell\}}(\bfQ,W_s) =H^1_{\Sigma'}(\bfQ, W_s).$$
\end{lemma}

\begin{proof}   First assume that $\ov{\xi}$ is ramified  at $\ell$. Then $W_1^{I_{\ell}}=0$ and so $W^{I_{\ell}}=0$ and we use \cite{BergerKlosin13} Lemma 5.6 to conclude that $$H^1_{\Sigma' \cup \{\ell\}}(\bfQ,W_s)=H^1_{\Sigma'}(\bfQ,W_s).$$ 

We note that the definition of the global Selmer group $H^1_\Sigma(\bfQ,W_s)$ in \cite{BergerKlosin13} differs from our definition here in that it uses the Fontaine-Laffaille condition at $p$, rather than assuming that classes are unramified. But on the level of divisible coefficients the definitions agree, so we apply \cite{BergerKlosin13} Lemma 5.6  to conclude $H^1_{\Sigma' \cup \{\ell\}}(\bfQ,W)=H^1_{\Sigma'}(\bfQ,W)$ and then invoke \eqref{functoriality}.

From now on assume that $\xi_s$ and so also $W_s$ is unramified at $\ell$. By  \cite{Rubin00}, Theorem 1.7.3 we have an exact sequence $$0 \to H^1_{\Sigma'}(\bfQ, W_s) \to H^1_{\Sigma' \cup \{\ell\}}(\bfQ, W_s) \to \frac{H^1(\bfQ_{\ell}, W_s)}{H^1_{\rm ur}(\bfQ_{\ell}, W_s)}.$$ Lemma 1.3.8(ii) in \cite{Rubin00} tells us that $H^1_{\rm ur}(\bfQ_{\ell}, W_s) = H^1_f(\bfQ_{\ell}, W_s)$. 

To prove the claim it is enough to show that the image of the map $H^1(\bfQ_{\ell}, W_s) \to H^1(I_{\ell}, W_s)$ is zero.  To do so consider the inflation-restriction sequence (where we set $G:=\Gal(\bfQ_{\ell}^{\rm ur}/\bfQ_{\ell})$):
$$H^1(G, W_s) \to H^1(\bfQ_{\ell}, W_s) \to H^1(I_{\ell}, W_s)^{G}\to H^2(G, W_s).$$ The last group in the above sequence is zero since $G\cong \hat{\bfZ}$ and $\hat{\bfZ}$ has cohomological dimension one. This means that the image of the restriction map $H^1(\bfQ_{\ell}, W_s) \to H^1(I_{\ell}, W_s)$ equals $H^1(I_{\ell}, W_s)^{G}$. Let us show that the  latter module is zero. Indeed, \begin{multline} H^1(I_{\ell}, W_s)^{G}=\Hom_G(I_{\ell}, W_s) = \Hom_G(I_{\ell}^{\rm tame}, W_s) \\= \Hom_G(\bfZ_p(1), \varpi^{-s}\Oo/\Oo( \xi^{-1} )) = \Hom_G(\bfZ_p, \varpi^{-s}\Oo/\Oo( \xi^{-1} \epsilon^{-1})).\end{multline} So, $\phi \in H^1(I_{\ell}, W_s)$ lies in $$H^1(I_{\ell}, W_s)^G=\Hom_G(\bfZ_p, \varpi^{-s}\Oo/\Oo( \xi^{-1} \epsilon^{-1}))$$ if and only if $\phi(x)=g \cdot \phi(g^{-1}\cdot x) = g\cdot \phi(x)=\xi_s^{-1} \epsilon^{-1}(g)\phi(x)$ for every $x \in I_{\ell}$ and every $g \in G$, i.e., if and only if \be \label{eq1} ( \xi_s^{-1} \epsilon^{-1}(g)-1)\phi(x)\in \Oo\quad \textup{for every $x \in I_{\ell}$, $g \in G$.}\ee Since $\Frob_{\ell}$ topologically generates $G$, we see that \eqref{eq1} holds if and only if it holds for every $x \in I_{\ell}$ and for $g=\Frob_{\ell}$.  So condition \eqref{eq1} becomes \be \label{eq2} (1-\xi_s^{-1}(\Frob_{\ell})\ell^{-1})\phi(x)\in \Oo\quad \textup{for every $x \in I_{\ell}$.}\ee Since $\ov{\xi}(\Frob_{\ell})\ell \neq 1$, the factor $\val_p(1-\xi_s^{-1}(\Frob_{\ell})\ell^{-1})=0$, we get that $\phi(x) \in \Oo$, as claimed.
\end{proof}

\section{Deformation theory}

\subsection{Assumptions} \label{The residual representation}

Let $p>2$ be a prime and $N $ a positive integer with $p\nmid N$.   Let $\tilde\chi: (\bfZ/Np\bfZ)^{\times} \to \bfC^{\times}$ 
denote a Dirichlet character of order prime to $p$ with $\tilde{\chi}(-1)=-1$. We write $\tilde{\chi}=\tilde{\chi}_N \tilde{\chi}_p$ where $\tilde{\chi}_N$ is a Dirichlet character mod $N$ and $\tilde{\chi}_p$ is a Dirichlet character mod $p$. We assume that  $\tilde{\chi}_N$ is primitive. In particular, we allow but do not require that $\tilde{\chi}$ has $p$ in its conductor.

Write $\Sigma$ for a finite set of primes containing $p$ and the primes dividing $N$. Let $\chi: G_{\Sigma} \to \Oo^{\times}$ be the Galois character associated to $\tilde\chi$ and write $\ov{\chi}: G_{\Sigma} \to \bfF^{\times}$ for its mod $\varpi$ reduction. 
We assume that $\ov{\chi}|_{D_p} \neq 1$.

Write $F:= \bfQ(\chi)$ for the splitting field of $\chi$ and ${\rm Cl}(F)$ for the class group of $F$. Set $C_F := {\rm Cl}(F)\otimes_{\bfZ} \Oo$. For any character $\psi: \Gal (F/\bfQ) \to \Oo^{\times}$ we write $C_F^{\psi}$ for the $\psi$-eigenspace of $C_F$ under the canonical action of $\Gal(F/\bfQ)$, i.e. $$C_F^{\psi}=\{c \in C_F| g \cdot c=\psi(g)c \, \text{ for all } g \in G_\Sigma\}.$$  In this paper we work under the following assumptions:
\begin{enumerate}
\item $C_F^{\chi^{-1}}$ is a non-zero cyclic $\Oo$-module, i.e., $\dim_{\bfF}C_F^{\chi^{-1}}\otimes_{\Oo}\bfF=1$;
\item if $\ell \in \Sigma$ but $\ell \nmid Np$ then   $\tilde\chi(\ell)\ell \not\equiv 1$ mod $\varpi$;
\item if $\ell \in \Sigma$ but $\ell \nmid Np$ then   $\tilde\chi(\ell)\not\equiv \ell$ mod $\varpi$.
\end{enumerate}

\begin{rem} We note that  $C_F^{\chi^{-1}}\neq 0$ is equivalent to $\val_p(L(0, \tilde \chi))>0$. This is so because  under our assumptions on $\chi$, we have that (cf. Theorem 2 in \cite{MazurWiles84}) \be \label{size of class group} \# C_F^{\chi^{-1}} = \# \Oo/L(0,\tilde \chi).\ee\end{rem}

 Let $\rho_0: G_{\Sigma} \to \GL_2(\bfF)$ be a continuous homomorphism of the form $$\rho_0 = \bmat 1 & * \\ 0 & \ov{\chi} \emat \not\cong 1 \oplus \ov{\chi}$$ such that $\rho_0|_{D_p}\cong 1\oplus  \ov{\chi}|_{D_p}$.

For the convenience of the reader we discuss in section \ref{summary} how the  assumptions are used.

\subsection{The residual representation} 
We begin by proving the uniqueness of $\rho_0$ up to isomorphism. Note that for this result we do not need to assume that $\rho_0$ is split on $D_p$, but only on $I_p$.

\begin{prop}\label{uniqueness} Let $\rho': G_{\Sigma} \to \GL_2(\bfF)$ be a continuous homomorphism of the form $$\rho'=\bmat 1 & * \\ 0 & \ov{\chi}\emat \not\cong 1 \oplus \ov{\chi}$$ such that $\rho'|_{I_p} \cong 1\oplus \ov{\chi}|_{I_p}$. Then $\rho'\cong \rho_0$. 
\end{prop}
\begin{proof} Let $\rho'$ be as in the statement of the proposition. Then $*$ gives rise to a non-zero element $c$  in $H^1_{\Sigma}(\bfQ, \bfF(\ov{\chi}^{-1}))$.  Using Lemma \ref{lower bound 21} and Assumption (2) above we conclude that $H^1_{\Sigma}(\bfQ, \bfF(\ov{\chi}^{-1}))=H^1_{\{p\}}(\bfQ, \bfF(\ov{\chi}^{-1}))$. By the assumption that $\rho'|_{I_p} \cong 1\oplus \ov{\chi}|_{I_p}$ we see that $c$ is unramified at $p$, hence in fact $c\in H^1_{\emptyset}(\bfQ, \bfF(\ov{\chi}^{-1}))$ by \eqref{urf}. 
By Proposition \ref{clgroup} we have that $H^1_{\emptyset}(\bfQ, E/\Oo(\chi^{-1}))\cong \Hom({\rm Cl}(F), E/\Oo(\chi^{-1}))^{\Gal(\bfQ(\chi)/\bfQ)}$. This last group is (non-canonically) isomorphic to  $C_F^{\chi^{-1}}$. By Assumption (1), the group $C_F^{\chi^{-1}}$ is cyclic, hence so is $H^1_{\emptyset}(\bfQ, E/\Oo(\chi^{-1}))$.
By \eqref{functoriality} we get that $H^1_{\emptyset}(\bfQ, \bfF(\ov{\chi}^{-1}))\cong H^1_{\emptyset}(\bfQ, E/\Oo(\chi^{-1}))[\varpi]$, so $H^1_{\emptyset}(\bfQ, \bfF(\ov{\chi}^{-1}))$ is also cyclic. Hence  the extension given by $c$ is a non-zero scalar multiple of the one given by $\rho_0$. The claim follows.
\end{proof}

\subsection{The deformation problems} \label{deformationproblem}

Set $R$ to be the universal deformation ring for deformations $\rho:G_\Sigma \to \GL_2(A)$ of $\rho_0$ for $A$  an object in ${\rm CNL}(\Oo)$, the category of local complete Noetherian $\Oo$-algebras with residue field $\bfF$, such that:
\begin{itemize}
\item[(i)] $\det \rho = \chi$
\item[(ii)] $\rho|_{D_p} \cong \bmat \psi_1 &* \\ & \psi_2\emat$ with $\psi_2$ unramified $\psi_2 \equiv 1 \mod{\fm_A}$ (ordinary and $p$-distinguished)

\item[(iii)] If $\ell \in \Sigma$ is such that $\ell \equiv 1$ (mod $p$) then $\rho|_{I_{\ell}} = 1 \oplus \chi$.

\end{itemize}

Let $\rho^{\rm univ}: G_{\Sigma} \to \GL_2(R)$ be the universal deformation. Write $I$ for the ideal of reducibility of $\rho^{\rm univ}$.

Let $R^{\rm split}$ be the universal deformation ring for the deformations where (ii) is strengthened to assuming that $\rho|_{D_p}$ is split. We denote the universal deformation for the stronger condition by $\rho^{\rm split}$ and write $I^{\rm split}$ for its ideal of reducibility.

We will refer to deformations satisfying (i)-(iii) as \emph{ordinary deformations} (or simply as \emph{deformations}), whilst calling the ones satisfying the stronger condition  \emph{split deformations}. It is clear that every split deformation is a deformation, so we get a natural map $R \to R^{\rm split}$.

\begin{rem} \label{CM1}
Note that by Corollary to Theorem 11 in \cite{Sen81}  the assumption that $\rho|_{D_p}$ is split corresponds to $\rho|_{D_p}$ being Hodge-Tate (and even de Rham) with Hodge-Tate weights $0$. Such representations $\rho$ are expected to correspond to classical modular forms of weight $1$.

If one knows that $\rho(G_{\bfQ})$ is finite then one can easily prove this special case of Artin's conjecture: From the classification of subgroups of ${\rm GL}_2(\bfC)$ one can show (see e.g. section 2 of \cite{DummiganSpencer}) that the residual reducibility requires the image of $\rho$ to be dihedral. From this one deduces (see e.g. section 7 in \cite{Serre77b}) that there exists a quadratic extension $F/\bfQ$ for which $\rho \otimes \chi_{F/\bfQ} \cong \rho$, where $\chi_{F/\bfQ}$ is the unique character of $G_{\bfQ}$ that factors through the non-trivial character of $\Gal(F/\bfQ)$. This implies that $\chi=\chi_{F/\bfQ}$, so $F$ has to be imaginary quadratic as $\chi$ is odd. It further follows that $\rho$ is the induction of a finite order Galois character of $G_F$, i.e. that $\rho$ corresponds to a weight 1 CM form.  In section \ref{sect8} we prove (without the assumption that $\rho(G_\bfQ)$ is finite) that split deformations of $\rho_0$ with $\chi=\chi_{F/\bfQ}$ indeed correspond to classical weight 1 CM forms.

\end{rem}

\subsection{Reducible deformations}

We record the following general lemma regarding pseudocharacters that helps us study reducible deformations.
\begin{lemma}\label{uniqueness of Ti}  Let $A$ be a Henselian local ring with a maximal ideal $\fm$ and let $G$ be a group. Let $\tau_1, \tau_2: G \to (A/\fm)^{\times}$ be two distinct characters which we can regard as homomorphisms from $A[G]$ to $A/\fm$. Let $T: A[G]\to A$ be a pseudocharacter of dimension 2 such that there exist characters $T_1, T_2$ with $T=T_1+T_2$ with the property that $T_i\otimes_AA/\fm=\tau_i$ for $i=1,2$. Then $T_1$ and $T_2$ are uniquely determined.
\end{lemma}

\begin{proof} This is the last assertion of Proposition 1.5.1 in \cite{BellaicheChenevierbook} where we take $J=0$ and $R=A[G]$ and then $\textup{dec}_{\mathcal{P}}$ is satisfied for $\mP=\{\{1\},\{2\}\}$ with $I_{\mP}=0$. \end{proof}

\begin{prop} \label{infi} There do not exist any non-trivial upper-triangular deformations of $\rho_0$ to $\GL_2(\bfF[X]/X^2)$.  
\end{prop}

\begin{proof} Suppose that $\rho: G_{\Sigma} \to \GL_2(\bfF[X]/X^2)$ is such a deformation. Write $$\rho = \bmat 1+ aX & b \\ & \ov{\chi} + dX\emat$$ with $a,d : G_{\Sigma} \to \bfF$ and $b: G_{\Sigma} \to \bfF[X]/X^2$. Our deformation conditions guarantee that $a$ and $d$ are unramified at all primes. Indeed, $a$ and $d$ are at most tamely ramified at all primes $\ell \neq p$, but  if $\ell \in \Sigma - \{p\}$ and $\ell \not\equiv 1$ mod $p$, then there is no abelian $p$-extension of $\bfQ$ that is tamely ramified at $\ell$. On the other hand if $\ell \equiv 1$ mod $p$ then the deformation condition (iii) guarantees that $a|_{I_{\ell}}=0$. So, $a$ can only be ramified at $p$.

 By condition (ii) we have an isomorphism of $\bfF[X]/X^2[D_p]$-modules \be \label{onIp}  \bmat 1+ aX & b \\ & \ov{\chi} + dX\emat\cong \bmat \psi_1 &* \\ & \psi_2\emat,\ee
where each of the entries is considered to be restricted to $D_p$ and $\psi_2\equiv 1$ mod $X$. 
Using Lemma \ref{uniqueness of Ti} we see that we therefore must have $1+aX=\psi_2$ as $\ov{\chi}|_{D_p} \neq 1$. As $\psi_2$ is unramified, we conclude that $a$ is unramified (at $p$). Hence $a$ is unramified everywhere, so $a=0$. By  condition (i) we must also have $d=0$.

Now consider the entry $b=b_0 + b_1X$ with $b_0, b_1: G_{\Sigma} \to \bfF$. Using the basis $\bmat 1\\ 0 \emat, \bmat 0 \\ 1 \emat, \bmat X\\0\emat, \bmat 0 \\ X \emat$ we can write $\rho$ as a 4-dimensional representation over $\bfF$: $$\rho =  \bmat 1 & b_0 \\ & \ov{\chi} \\ &b_1 & 1 & b_0 \\ &&&\ov{\chi}\emat$$ which clearly has a subquotient isomorphic to $\bmat 1 & b_1 \\ & \ov{\chi}\emat$. If this subquotient is split we are done. Otherwise it must be isomorphic to $\rho_0$ by Proposition \ref{uniqueness}. From this it is easy to see that $\rho'\cong \rho_0$ as desired (cf. Proof of Proposition 7.2 in \cite{BergerKlosin13} for details).  
\end{proof}

\begin{cor} \label{structure} The structure maps $\Oo \to R/I$ and $\Oo \to R^{\rm split}/I^{\rm split}$ are surjective. \end{cor}
\begin{proof} Using Proposition \ref{infi} this is proved like Proposition 7.10 in \cite{BergerKlosin13}. \end{proof}
As a consequence of Corollary \ref{structure} one gets as in Proposition 7.13 in \cite{BergerKlosin13} the following proposition.

\begin{prop} \label{genbytraces} The ring $R$ is topologically generated as an $\Oo$-algebra by the set $\{\tr \rho^{\rm univ}(\Frob_{\ell}) \mid \ell \not\in\Sigma\}$ and $R^{\rm split}$ is topologically generated by $\{\tr \rho^{\rm split}(\Frob_{\ell}) \mid \ell \not\in\Sigma\}$. 
\end{prop}

\begin{prop} \label{bound on R/I} One has $\# R^{\rm split}/I^{\rm split} \leq \# R/I \leq \#C_F^{\chi^{-1}}$.  
\end{prop}
\begin{proof} 
Note that Proposition \ref{genbytraces}  implies that the map $R \to R^{\rm split}$ is a surjection. This implies (see e.g. \cite{BergerKlosin13} Lemma 7.11) that the map $R/I \to R^{\rm split}/I^{\rm split}$ is also surjective, so we only need to prove the second inequality.

By Corollary \ref{structure} we get $R/I = \Oo/\varpi^r$ (allowing for $r=\infty$). Using Corollary 7.8 in \cite{BergerKlosin13} we know that any deformation to $\GL_2(R/I)$ is equivalent to one of the form $\bmat \Psi'_1 & b' \\ & \Psi'_2\emat $ with $\Psi'_1$ reducing to $1$ mod $\varpi$ and $\Psi'_2$ reducing to $\ov{\chi}$ mod $\varpi$.  (Note that the corollary assumes that the ring $R/I$ is Artinian. However, its proof uses Theorem 7.7 which allows for the ring to be Hausdorff and complete, so the case of  $r=\infty$ is also covered.) 
Let $s\leq r$ be a (finite) positive integer. Let $\rho: G_{\Sigma} \to \GL_2(\Oo/\varpi^s)$ be the composition of the deformation $\bmat \Psi'_1 & b' \\ & \Psi'_2\emat $  with the canonical projection $R/I \twoheadrightarrow \Oo/\varpi^s$. Then $\rho = \bmat \Psi_1 &  b\\& \Psi_2\emat$, where the non-primed entries are simply the reductions of the primed entries modulo $\varpi^s$. 
Write $\Psi_1=1+\alpha \varpi$ for some group homomorphism $\alpha: G_{\Sigma} \to \Oo/\varpi^{s-1}$. Hence $\Psi_1$ cuts out an abelian extension $K$ of $\bfQ$ that is of $p$-power degree. Let $\ell \in \Sigma$ be a prime different from $p$. Then $K$ can be at most tamely ramified at $\ell$, so it must be unramified unless $\ell \equiv 1 $ mod $p$. So the deformation condition (iii) guarantees that $K$ can only be ramified at $p$.  

By condition (ii) we get that $\rho|_{D_p}\cong \bmat \psi_1 & * \\ &\psi_2\emat$ with $\psi_2$ unramified at $p$ and reducing to the trivial character mod $\varpi$.
Using Lemma \ref{uniqueness of Ti} we must therefore have that $\Psi_1|_{D_p}=\psi_2$ as $\ov{\chi}|_{D_p}\neq 1$. Hence $\alpha$ must be unramified at $p$. 
Thus we have shown that  $\Psi_1$ is unramified everywhere and hence $\Psi_1=1$. Then condition (i) implies that $\Psi_2=\chi$.

We thus get that $b$ gives rise to a cohomology class in $H^1_{\Sigma}(\bfQ, \Oo/\varpi^s(\chi^{-1}))=H^1_{\Sigma}(\bfQ, W_s)$, where $W=E/\Oo(\chi^{-1})$. Using Lemma \ref{lower bound 21} and Assumption (2) we see that this  group equals $H^1_{\{p\}}(\bfQ, W_s)$. Condition (ii) now again forces $b$ to be unramified at $p$ as well, so in fact the class of $b$ lies in $H^1_{\emptyset}(\bfQ, W_s)$.

By \eqref{functoriality} we have $H^1_{\emptyset}(\bfQ, W_s)=H^1_{\emptyset}(\bfQ,W)[\varpi^s]$  and by Proposition \ref{clgroup} we have a non-canonical isomorphism
 $H^1_{\emptyset}(\bfQ, W) \cong C^{\chi^{-1}}_F$ (cf. the proof of Proposition \ref{uniqueness}). 

Define $k$ by $\#C_F^{\chi^{-1}}\cong \#\Oo/\varpi^k$. Then we conclude that $\varpi^k$ annihilates the class in $H^1_{\emptyset}(\bfQ, W)$ arising from $b$.
As $b$ is not a coboundary mod $\varpi$ by the assumption that $\rho_0$ is not split, we get that the image of $b$ in $\Oo/\varpi^s$ generates $\Oo/\varpi^s$ over $\Oo$. So, the class of $b$ generates an $\Oo$-submodule of $H^1_{\emptyset}(\bfQ, W)$ isomorphic to $\Oo/\varpi^s$. Hence $s\leq k$. If $r<\infty$, we can always take $s=r$, so this forces also $r\leq k$. If $r=\infty$, we could take $s=k+1$, which would lead to a contradiction, so $r$ cannot be infinite. 
\end{proof}

\subsection{Principality of ideals of reducibility}
Let $\tilde{\omega}:(\bfZ/p\bfZ)^{\times}\to \bfC^{\times}$ be the Teichmüller character. We denote by $\omega:G_{\Sigma} \to \bfZ_p^{\times}$  the corresponding $p$-adic Galois character. 

In this section we will prove the following result.

\begin{thm} \label{principality}
\,
\begin{enumerate} \item Suppose $C_F^{\chi}$ is a cyclic $\Oo$-module, then   the ideal $I^{\rm split}$ is principal. 

\item  Suppose that $C_F^{\chi}=0$. Assume further that at least one of the following conditions is satisfied: \begin{itemize}
\item[(i)] $e<p-1$ where $e$ is the ramification index of $p$ in $\bfQ(\chi)$ or
\item[(ii)] $\chi=\omega^s$ for some integer $s$ or \item[(iii)] $\tilde{\chi}_N(p)\neq 1$.\end{itemize} Then $I$ is principal.
\end{enumerate}
\end{thm}

\begin{rem}\label{exclusions} \,
\begin{enumerate} 
\item[(i)]  If $\chi$ is quadratic then $C_F^{\chi}=C_F^{\chi^{-1}}$, so the assumption in part (1) of Theorem \ref{principality} follows from Assumption (1) in section \ref{deformationproblem}.
\item[(ii)] Note that  $\chi$ in part (2) of the Theorem is automatically non-quadratic as we assume that $C_F^{\chi}=0$ while we have $C_F^{\chi^{-1}}\neq 0$. 
\item[(iii)] Let $\chi_N$ (resp. $\chi_p$ for later usage) be the Galois character associated with $\tilde{\chi}_N$ (resp. $\tilde{\chi}_p$).  Part (2) of Theorem \ref{principality} does not cover the case where $\chi=\omega^s \chi_N$ where $(s,p-1)=1$, so $e=p-1$, but $\chi_N$ is a non-trivial character with $\chi_N(p)=1$, which means that $\bfQ(\chi)$ is an extension of $\bfQ(\zeta_p)$ where all primes of $\bfQ(\zeta_p)$ lying over $p$ split completely in $\bfQ(\chi)/\bfQ(\zeta_p)$.
\end{enumerate} \end{rem}

\begin{proof} The universal deformations give rise to $R$-algebra homomorphisms $\rho:=\rho^{\rm univ}: R[G_{\Sigma}]\to M_2(R)$ and $\rho^{\rm split}: R^{\rm split}[G_{\Sigma}]\to M_2(R^{\rm split})$.  
 Fix $?\in \{\emptyset, {\rm split}\}$. 
The image of $\rho^{?}$ is a Generalized Matrix Algebra (GMA) in the sense of \cite{BellaicheChenevierbook} of the form $$\bmat R^? & B^? \\ C^? & R^? \emat,$$ where $B^?$ is the ideal of $R^?$ generated by $b^?(x)$ as $x$ runs over $R^?[G_{\Sigma}]$ and similarly for $C^?$.  As the residual representation is non-split we get $B^?=R^?$, so we get $I^?=B^?C^?=C^?$. Arguing as in the proof of Theorem 1.5.5. in \cite{BellaicheChenevierbook}, and using the fact that $I^?\subset \fm^?$ (where $\fm^?$ is the maximal ideal of $R^?$)  we get an injection:
$$\iota^?: \Hom_{R^?}(C^?, R^?/\fm^?)=\Hom_{R^?}(C^?, \bfF) \hookrightarrow H^1(\bfQ, \bfF(\ov{\chi})).$$

We first claim that the image lands inside $H^1_{\{p\}}(\bfQ, \bfF(\ov{\chi}))$, i.e., that it consists only of classes that are unramified outside $p$. First note that the map $\iota^?$ is given by (cf. Proof of Theorem 1.5.5 in \cite{BellaicheChenevierbook}) $$f \mapsto \left(x \mapsto \bmat a^?(x) \pmod{\fm^?} & 0\\  f(c^?(x))& d^?(x)\pmod{\fm^?}\emat\right).$$ Then it is clear that the image of $\iota^?$ is contained in $H^1_{\Sigma}(\bfQ, \bfF(\ov{\chi}))$, i.e., is unramified outside $\Sigma$. We will show that in both cases of the Theorem, this image is a one-dimensional $\bfF$-vector space. This implies that $I^?$ is a principal ideal, 
as we now explain.

Indeed, if the image of $\iota^?$ is one-dimensional, so is $\Hom_{R^?}(C^?, R^?/\fm^?)$.  The natural injection $$ \Hom_{R^?/\fm^?}(C^?/\fm^?C^?, R^?/\fm^?) =\Hom_{R^?}(C^?/\fm^?C^?, R^?/\fm^?) \hookrightarrow \Hom_{R^?}(C^?, R^?/\fm^?)$$  is an isomorphism, so this forces the $\bfF$-vector space $C^?/\fm^? C^?$ to be one-dimensional. Hence $C^?$ is a cyclic $R^?$-module by the complete version of the Nakayama's Lemma. We conclude that $C^? \cong R^?$ as $R^?$-modules, so $I^?$ is principal.

So, it remains to prove the one-dimensionality of the image of $\iota^?$. By Lemma \ref{lower bound 21} applied with $\tilde{\xi}=\chi^{-1}$  and $s=1$ we see that $H^1_{\Sigma}(\bfQ, \bfF(\ov{\chi}))=H^1_{\Sigma'}(\bfQ, \bfF(\ov{\chi}))$
 where $\Sigma'\subset \Sigma$ consists only of $p$ and those primes $\ell$ such that $\chi$  is unramified at $\ell$ and $\tilde\chi(\ell) \equiv\ell$ mod $p$. If $\ell$ is one of the latter primes then by our assumption (3) $\rho^{?}$ is unramified at $\ell$. This implies that the image of $\iota^?$ is contained in $H^1_{\{p\}}(\bfQ, \bfF(\ov{\chi}))$.
 
Now suppose we are in the case (1) of the Theorem. Then $?={\rm split}$ and the image of $\iota^{\rm split}$ is in fact contained in $H^1_{\emptyset}(\bfQ, \bfF(\ov{\chi}))$. Arguing as in the proof of Proposition \ref{uniqueness} we see that the one-dimensionality of this Selmer group is equivalent to the $\Oo$-cyclicity of  $C_F^{\chi}$. This proves part (1) of the Theorem.

From now on we study case (2) when $?=\emptyset$ and we will show that $$\dim_{\bfF}H^1_{\{p\}}(\bfQ, \bfF(\ov{\chi}))\leq 1.$$

As by Remark \ref{invariants1}, the character $\chi$ is the Teichmüller lift of $\ov{\chi}$, to ease notation below we will not distinguish between $\chi$ and $\ov{\chi}$ and always write $\chi$.
 Write $G$ for $\Gal(\bfQ(\chi)/\bfQ)$.
Consider the inflation-restriction exact sequence 
$$H^1(G, \bfF(\chi)^{\ker \chi})\to H^1(G_{\Sigma}, \bfF(\chi)) \to H^1(\ker \chi, \bfF(\chi))^G \to H^2(G, \bfF(\chi)^{\ker \chi}).$$
As $\chi$ has order prime to $p$, we see that $G$ has order prime to $p$, so the first and the last group must be zero as $\bfF(\chi)$ is killed by $p$. Hence the restriction maps gives us an isomorphism 
\be \label{infres3} {\rm res}: H^1(G_{\Sigma}, \bfF(\chi)) \cong H^1(\ker \chi, \bfF(\chi))^G,\ee where the last group equals $\Hom_G((\ker \chi)^{\rm ab}, \bfF(\chi))$. 
The isomorphism \eqref{infres3} carries classes unramified outside of $p$ to classes unramified outside $p$, which then correspond to homomorphisms in  $\Hom_G((\ker \chi)^{\rm ab}, \bfF(\chi))$ that are trivial on all inertia groups $I_{\ell}$ for all primes $\ell \neq p$. 
Hence the group $H^1_{\{p\}}(\bfQ, \bfF(\chi))$ maps into the subgroup of $\Hom_G((\ker \chi)^{\rm ab}, \bfF(\chi))$ consisting of all the homomorphisms which vanish on all $I_{\ell}$ for $\ell \neq p$. If we denote by $H$ the image of $(\ker \chi)^{\rm ab}$ in the group $G_{\{p\}}$ under the canonical map $G_{\Sigma} \twoheadrightarrow G_{\{p\}}$, then each of these homomorphisms factors through $H$. So they land in the subgroup $\Hom_G(H, \bfF(\chi))$ (which injects into $\Hom_G((\ker \chi)^{\rm ab}, \bfF(\chi))$ by left exactness of the $\Hom$-functor).

 Furthermore, as each element of $ \Hom_G(H, \bfF(\chi))$ is annihilated by $p$, we get that $\Hom_G(H, \bfF(\chi)) \cong \Hom_G(V, \bfF(\chi))$, where $V=H/H^p$. We can identify $V$ with (a quotient of) the Galois group $\Gal(L/\bfQ(\chi))$ where $L$ is the maximal abelian extension of $\bfQ(\chi)$ which is annihilated by $p$ and unramified away from primes of $\bfQ(\chi)$ lying over $p$. As one has  $\dim_{\bfF} \Hom_G(V, \bfF(\chi))=\dim_{\ov{\bfF}_p} \Hom_G(V, \ov{\bfF}_p(\chi))=\dim_{\ov{\bfF}_p}\Hom_{\bfF_p[G]}(V, \ov{\bfF}_p(\chi))=\dim_{\ov{\bfF}_p}\Hom_{\ov{\bfF}_p[G]}(V\otimes_{\bfF_p}\ov{\bfF}_p, \ov{\bfF}_p(\chi))$
 it suffices to prove that $$\dim_{\ov{\bfF}_p} \Hom_{\ov{\bfF}_p[G]}(V\otimes_{\bfF_p}\ov{\bfF}_p, \ov{\bfF}_p(\chi))\leq 1.$$

One has $$V\otimes_{\bfF_p}\ov{\bfF}_p=\bigoplus_{\varphi\in \Hom(G, \ov{\bfF}^{\times})}V^{\varphi},$$ where $$V^{\varphi}=\{v\in V\otimes_{\bfF_p} \ov{\bfF}_p\mid g\cdot v=\varphi(g) v \hs \textup{for every $g\in G$}\}.$$ 
It is clear that $$\Hom_{\ov{\bfF}_p[G]}(V\otimes_{\bfF_p}\ov{\bfF}_p, \ov{\bfF}_p(\chi))\cong \Hom_{\ov{\bfF}_p[G]}(V^{\chi}, \ov{\bfF}_p(\chi)).$$ 
Hence it suffices to show that $$\dim_{\ov{\bfF}_p} V^{\chi}\leq 1.$$

Write $S$ for the set of primes of $F=\bfQ(\chi)$ lying over $p$. Set $\Oo$ to be the ring of integers in $F$. For $\fp\in S$ let $\Oo_{\fp}$ denote the completion of $\Oo$ at $\fp$.  Let $M=\prod_{\fp\in S} (1+\fp\Oo_{\fp})$. Let $\mE$ denote the image of the global units of $F$ in $M$ and $\ov{\mE}$ the closure of $\mE$ in $M$. Using the assumption that $C_F^{\chi}=0$ by Corollary 13.6 in \cite{Washingtonbook} we get that $$M/\ov{\mE}\cong\Gal(K/H),$$ where $K$ denotes the maximal abelian pro-$p$ extension of the Hilbert class field $H$  of $F$ unramified outside of $S$. 

We have an exact sequence $$0\to M/\ov{\mE} \to \Gal(K/\bfQ(\chi)) \to \Cl(F)\to 0.$$ Tensoring with $\ov{\bfF}_p$ we get $$M/\ov{\mE}\otimes_{\bfZ_p}\ov{\bfF}_p \to V\otimes_{\bfF_p}\ov{\bfF}_p \to \Cl(F)\otimes_{\bfZ_p}\ov{\bfF}_p\to 0.$$ Using the assumption that $C_F^{\chi}=0$, we see that $(M/\ov{\mE}\otimes_{\bfZ_p}\ov{\bfF}_p)^{\chi}$ surjects onto $V^{\chi}$. As tensoring is right exact the surjection $M\to M/\ov{\mE}$ gives rise to a surjection $M\otimes_{\bfZ_p}\ov{\bfF}_p \to M/\ov{\mE}\otimes_{\bfZ_p}\ov{\bfF}_p$, hence $V^{\chi}$ is a quotient of $M\otimes_{\bfZ_p}\ov{\bfF}_p$.

It suffices to show that $\dim_{\ov{\bfF}_p}(M/\ov{\mE}\otimes_{\bfZ_p}\ov{\bfF}_p)^{\chi}\leq  1$. This would follow if we can show that $\dim_{\ov{\bfF}_p}(M\otimes_{\bfZ_p}\ov{\bfF}_p)^{\chi}\leq 1$. What we will actually show is that under our assumptions we have \be \label{MT1} \dim_{\ov{\bfF}_p}(M/T\otimes_{\bfZ_p}\ov{\bfF}_p)^{\chi}=1,\ee  where $T$ denotes the torsion submodule of $M$.  This is sufficient. Indeed, if $e<p-1$, we will show that $T=0$, hence \eqref{MT1} implies that $\dim_{\ov{\bfF}_p}(M\otimes_{\bfZ_p}\ov{\bfF}_p)^{\chi}= 1$. In the case when $\tilde\chi_N(p)\neq 1$, we show that $\chi$ does not occur in $T\otimes_{\bfZ_p}\ov{\bfF}_p$, hence again \eqref{MT1} gives us $\dim_{\ov{\bfF}_p}(M\otimes_{\bfZ_p}\ov{\bfF}_p)^{\chi}=1$. In the case when $\chi=\omega^s$, it can occur in $T$, but we will then show that $T\subset \ov{\mE}$, so proving $\dim_{\ov{\bfF}_p}(M/T\otimes_{\bfZ_p}\ov{\bfF}_p)^{\chi}=1$ again suffices for demonstrating that $\dim_{\ov{\bfF}_p}(M/\ov{\mE}\otimes_{\bfZ_p}\ov{\bfF}_p)^{\chi}\leq 1$.

Note that $M/T$ is a free $\bfZ_p$-module of rank $|G|$. We now analyze the action of $G$ on $T$. Note that $T$ results from the $p$-adic logarithm map not being an isomorphism on $1+\fp$. Combining Proposition II.5.5 and the proof of Proposition II.5.7(i) in \cite{Neukirch99} we see that $\log_p$ is an isomorphism as long as the ramification index $e$ of $p$ in $F$ is less than $p-1$. Hence if $e<p-1$, then $T=0$. 

For $\fp \in S$ we write $G_\fp$ for the stabilizer of $\fp$ in $G$. As $\chi=\chi_p\chi_N$ and $\chi_N$ is unramified at $p$, we get that $e\leq p-1$. In fact, $\bfQ(\chi_p)\subset \bfQ(\mu_p)$, so $e=p-1$ if and only if $\bfQ(\chi_p)=\bfQ(\mu_p)$. Hence we conclude that $T\neq 0$ if and only if $\bfQ(\chi_p)=\bfQ(\mu_p)$. So, assume $\bfQ(\chi_p)=\bfQ(\mu_p)$. Then we have $$1+\fp\Oo_{\fp}\cong \bfZ_p^{\#G_{\fp}}\times \mu_{p^a}.$$  Hence $$T\cong (\mu_{p^a})^{\#S}$$  and $T\otimes_{\bfZ_p}\ov{\bfF}_p=(\mu_{p^a})^{\#S} \otimes_{\bfZ_p} \bfF_p \otimes_{\bfF_p}\ov{\bfF}_p=(\mu_{p})^{\#S}  \otimes_{\bfF_p}\ov{\bfF}_p$. As $\mu_p \not\subset \bfQ_p$, the group $G_{\fp}$ acts on the corresponding copy of $\mu_p$ via $\omega$ and $G$ acts on $T\otimes_{\bfF_p}\ov{\bfF}_p$ by the character $\omega \cdot \psi$ where $\psi$ is an $\#S$-cycle (as $G$ is cyclic) in the symmetric group on $\#S$ letters. Suppose that $\chi$ occurs in $T\otimes_{\bfF_p}\ov{\bfF}_p$. Then we must have $\chi_p=\omega$ and $\chi_N=\psi$. As the order of $\psi$ is $\#S$, we get that $\bfQ(\chi_N)/\bfQ$ has  degree $\#S$. As $p$ splits into $\#S$ primes in $\bfQ(\chi_N)$ we conclude that $p$ splits completely in $\bfQ(\chi_N)$ which is equivalent to saying that $\tilde\chi_N(p)=1$. This shows that if $\chi$ appears in $T\otimes_{\bfF_p}\ov{\bfF}_p$ then $\tilde\chi_N(p)=1$. 

If $\chi=\omega^s$ there is only one prime above $p$ and either $e<p-1$
(which implies $T=0$) or $T=\mu_p$. However, in the latter case also $\mE$ contains a copy of $\mu_p$, so in the quotient $M/\ov{\mE}\otimes \ov{\bfF}_p$, the torsion part $T$ gets annihilated. 

Hence, as explained before, it now suffices to show \eqref{MT1}.
Note that to decompose $(M/T)
\otimes_{\bfZ_p}
\ov{\bfF}_p$ it is
enough to decompose $\prod_{\fp \in S} \fp\Oo_{\fp}
\otimes_{\bfZ_p} \ov{\bfF}_p$,
since
$(1+\fp\Oo_{\fp})/(\textup{torsion}) \cong
\fp\Oo_{\fp}$
as $\bfZ_p[G_{\fp}]$-modules. One has  $\dim_{\ov{\bfF}_p}\fp\Oo_{\fp}\otimes_{\bfZ_p}\ov{\bfF}_p=t$, where $t=|G_{\fp}|$ and
$\dim_{\ov{\bfF}_p}\prod_{\fp \in S} \fp \Oo_{\fp} \otimes_{\bfZ_p}\ov{\bfF}_p=r$, where $r=|G|$.

 As $|\Hom(G, \ov{\bfF}_p^{\times})|=r$, it suffices to show that $(M/T\otimes_{\bfZ_p}\ov{\bfF}_p)^{\varphi}\neq 0$ for all $\varphi \in  \Hom(G, \ov{\bfF}_p^{\times})$.
Note that $G$ being isomorphic to the image of $\chi$, hence to a subgroup of $\bfF^{\times}$, is a cyclic group. If we denote by $\zeta$ a  primitive $r$th root of unity in $\ov{\bfF}_p$ then the characters $g\mapsto \zeta^i$ for $i=0,1,\dots, r-1$ exhaust all the characters in $\Hom(G,\ov{\bfF}^{\times}_p)$.

Let $\alpha \in \fp\Oo_{\fp}$ be such that $\{g^i\alpha \mid i=0,1,...,t-1\}$ is linearly independent where $g$ is a generator of $G_{\fp}$. This is possible because the extension $F/\bfQ$ has degree prime to $p$, so  is at most tamely ramified at $p$, hence the ideal $\fp$ possesses a normal integral basis - cf. Theorem 1 in \cite{Ullom}.  

We now claim that the set $\{\gamma \alpha \mid \gamma \in G\}$ is a linearly independent set in $\prod_{\fp \in S} \fp \otimes_{\bfZ_p}\ov{\bfF}_p$. Indeed, if $x\in \fp'\Oo_{\fp'}$ for some prime $\fp'$ of $F$ over $p$, and $\delta \in G$ is an element such that $\delta x\in \fp\Oo_{\fp}$, then by the above we can write 
$$\delta x=a_0 \alpha + a_1 g\alpha + \dots + a_{t-1} g^{t-1} \alpha \quad \textup{for some $a_0, a_1, \dots, a_{t-1} \in \bfZ_p$}.$$ Hence we conclude that $$x=a_0\delta^{-1} \alpha + a_1\delta^{-1} g\alpha + \dots + a_{t-1}\delta^{-1} g^{t-1}.$$ This shows that there exist $\gamma_1,\gamma_2,\dots, \gamma_t \in G$ such that $\{\gamma_1\alpha, \dots, \gamma_t\alpha\}$ is a $\ov{\bfF}_p$-basis of $\fp'\otimes_{\bfZ_p}\ov{\bfF}_p$,  hence our claim is proved.

 With this we fix a generator $g$ of $G$ and observe that for each $i\in \{0,1,\dots, r-1\}$ the vector 
 $$v_i=\alpha +\zeta^{-i}g\alpha +\zeta^{-2i}g^2\alpha +\dots+\zeta^{-(r-1)i}g^{r-1}\alpha$$  is an eigenvector for the action of $G$ on which $G$ acts via the character $g\mapsto \zeta^i$.  
\end{proof}

\section{$R=T$ Theorem in the ordinary case} 
Our  methods for proving an $R=T$ theorem in the split and the ordinary case are different. The ordinary case will be treated in this section  using Wiles's Theorem on the $\Lambda$-adic Eisenstein congruences and $T$ will be a Hecke algebra acting on non-classical weight one cusp forms. In the split case (treated in section \ref{sect8}) we will construct  Eisenstein congruences with classical weight one cusp forms directly (i.e., without using Wiles' result) - see also Remark \ref{rem5.5}. %for a comparison of these cases.
\subsection{$\Lambda$-adic Eisenstein congruences} Let $\theta: (\bfZ/Np\bfZ)^\times \to \bfC^\times$ be a primitive even Dirichlet character. Let $\Oo' \subset \ov{\bfQ}_p$ be the valuation ring of any finite extension of $\bfQ_p$ containing the values of $\theta$. Put $\Lambda=\Oo'[[T]]$ and let $\fX=\{(k, \zeta)\mid k \in \bfZ, k \geq 1, \zeta \in \mu_{p^{\infty}}\}$. For every $(k, \zeta) \in \fX$ we have an $\Oo'$-algebra homomorphism $\nu_{k, \zeta}: \Lambda \to \Oo'[\zeta]$ induced by $\nu_{k, \zeta}(1+T)=\zeta u^{k-2}$ where $u=\epsilon(\gamma)$ for $\gamma$ a topological generator of $\Gal(\bfQ_{\infty}/\bfQ)$ and $\epsilon: \Gal(\bfQ_{\infty}/\bfQ)\xrightarrow{\sim} 1+p\bfZ_p$ the $p$-adic cyclotomic character. Here $\bfQ_{\infty}$ is the unique $\bfZ_p$-extension of $\bfQ$. Note that under our assumption on the conductor of $\theta$ we have $\bfQ(\theta) \cap \bfQ_{\infty}=\bfQ$.

We fix an algebraic closure $\ov{F}_{\Lambda}$ of $F_{\Lambda}$, the fraction field of $\Lambda$, and regard all finite extensions of $F_{\Lambda}$ as embedded in that algebraic closure.
For $L\subset \ov{F}_{\Lambda}$ a finite extension of $F_{\Lambda}$, and $\Oo_L$ the integral closure of $\Lambda$ in $L$ put $\fX_L=\{\varphi: \Oo_L \to \ov{\bfQ}_p \text{ extending some } \nu_{k, \zeta}\}$.

We define an $\Oo_L$-adic modular form of tame level $N$ and character $\theta$ to be a collection of Fourier coefficients $c(n, \mF)$, $n\in \bfZ_{\geq 0}$ with the property that for all but finitely many $\varphi \in \fX_L$ extending $\nu_{k, \zeta}$  with $(k, \zeta) \in \fX$ and $\zeta$ of exact order $p^{r-1}$ for $r \geq 1$ there is an element $\mF_{\varphi}\in M_k(Np^r, \theta \tilde{\omega}^{2-k} \chi_{\zeta}, \varphi(\Oo_L))$ whose $n$th Fourier coefficient equals $\nu_{k, \zeta}(c(n, \mF))$. Here $\chi_\zeta$ is the Dirichlet character defined by mapping the image of $1+p$ in $(\bfZ/p^r\bfZ)^\times\cong (\bfZ/p\bfZ)^\times \times \bfZ/p^{r-1}\bfZ$ to $\zeta$. $\mF$ is called a \emph{cusp form} if all the $\mF_{\varphi}$ are cusp forms.  We denote by $\mM_{\OL}(N, \theta)$ the $\Oo_L$-torsion free module consisting of $\Oo_L$-adic modular forms having character $\theta$ and we set $\mM_{L}(N, \theta)=\mM_{\OL}(N, \theta)\otimes_{\OL} L$  and similarly for $\mS_L(N, \theta)$. 
This module has a natural action of Hecke operators and we denote by $\mM^0_{\OL}(N, \theta)$ the submodule $e\mM_{\OL}(N, \theta)$ cut out by applying the Hida ordinary projector $e$.  The corresponding subspace of cusp forms will be denoted by $\mS^0_{\OL}(N, {\theta})$.

Let $\bfT$ denote the $\Lambda$-algebra generated by all the Hecke operators $T_n$, $n \in \bfZ_+$ acting on $\mS^0_{{\Lambda}}({N}, \theta)$. 
By a result of Hida (Theorem 3.1  in \cite{Hida86a})  $\bfT$ is finitely generated and free as a $\Lambda$-module.
Newforms are then defined in an obvious way,  see \cite{Wiles88} section 1.5.

Fix $L\subset \ov{F}_{\Lambda}$ to be a finite extension of $F_{\Lambda}$ over which all newforms in $\mS^0_{{{\Lambda}}}({N}, {\theta}) \otimes_\Lambda \ov{F}_{\Lambda}$ are defined. 
Let $\mN'$ be the set of all newforms in $\mS^0_L(N, \theta)$ and fix a complete set $\mS'\subset \mN'$ of representatives of the Galois conjugacy classes  (over $F_{\Lambda}$) of all the elements of $\mN'$.
For $\mF\in \mN'$, if we denote by $L_{\mF}$ the extension of $F_{\Lambda}$ generated by the Fourier coefficients of $\mF \in  \mS^0_{{{\Lambda}}}({N}, {\theta}) \otimes_\Lambda \ov{F}_{\Lambda}$  then $\bfT$ can be naturally  viewed (by mapping an operator $t$ to the tuple $(c(1,t\mF))_{\mF}$) as a subring of the $F_{\Lambda}$-algebra $\prod_{\mF\in \mS'} L_{\mF}$ 
and one has $\bfT\otimes_{\Lambda}F_{\Lambda} =\prod_{\mF\in \mS'} L_{\mF}$ (cf. \cite{Wiles90}, eq. (4.1)). In fact, we have $\bfT \subset \prod_{\mathcal{F}\in \mS'} \Oo_{L_{\mF}}$ as $c(1,t\mF)$ are integral over $\Lambda$,  see e.g. \cite{Wiles88} p. 546.
\begin{definition} \label{cell} For each prime $\ell \neq p$ put $c_{\ell}:=1+\theta(\ell) \ell (1+T)^{a_{\ell}},$ where $a_{\ell} \in \bfZ_p$ is defined by $\ell=\tilde \omega(\ell) (1+p)^{a_{\ell}}$. Put $c_p:=1$. \end{definition}
These are the Hecke eigenvalues of a $\Lambda$-adic Eisenstein series  with constant term given by $L_p(s, \theta)/2$.
Here the Kubota-Leopoldt $p$-adic $L$-function $L_p(s, \theta)$ is an analytic function for $s \in \bfZ_p - \{1 \}$ (and even at $s=1$ if $\theta \neq 1$), which satisfies the interpolation property \be \label{Lpinterpolation} L_p(1-k, \theta)=(1-\theta \tilde \omega^{-k}(p) p^{k-1}) L(1-k, \theta \tilde \omega^{-k})\ee for $k \in \bfZ_{\geq 1}$. Iwasawa showed that there exists a unique power series $G_\theta(T) \in \Lambda$ such that $$L_p(1-s, \theta)=G_\theta(u^s-1).$$ Note that in general there is a denominator $H_\theta$ but for us it is identically 1 since $\theta$ is of type S (in the sense that $\bfQ(\theta) \cap \bfQ_{\infty}=\bfQ$).

Put $\hat{G}_\theta(T)=G_{\theta \tilde \omega^2}(u^2(1+T)-1)$ and $\hat{G}^0_\theta=\pi^{- \mu} \prod_{\zeta \in \mu_{p^\infty}} (1+T-\zeta u^{-1})^{-s_\zeta} \hat{G}_\theta(T)$, where $\pi$ is a uniformizer of $\bfZ_p[\theta]$. Here $\pi^{\mu}$ (respectively  $(1+T-\zeta u^{-1})^{s_\zeta}$) is the highest power of $\pi$ (respectively $ (1+T-\zeta u^{-1})$) common to all coefficients of $\hat{G}_\theta$. 

\begin{definition}
Define the Eisenstein ideal $J\subset \bfT$ to be the ideal generated by $T_\ell-c_{\ell}$ for all primes $\ell$ and by $\hat{G}^0_{\theta}(T)$. 
\end{definition}

We have the following result due to Wiles.
\begin{thm} [Wiles, \cite{Wiles90}, Theorem 4.1] \label{Wiles2} If $\theta \neq \tilde{\omega}^{-2}$ then one has $$\bfT/J \cong \Lambda/\hat{G}^0_{\theta}(T).$$ \end{thm}

Let $\theta=\tilde{\chi} \tilde{\omega}^{-1}$ 
and put $\Oo'=\Oo$. (Note that the values of $\omega$ are already contained in $\bfZ_p$.)

\begin{rem} \label{r4.4} Note that  the  theorem rules out $\chi=\omega^{-1}$ while Lemma \ref{noexceptionalzeroes} below  rules out $\chi=\omega$. However, in both cases $C_F^{\chi^{-1}}=0$ by \cite{Washingtonbook} Proposition 6.16 and Theorem 6.17, so these cases are not relevant for our deformation problem as we assume in section \ref{The residual representation} that $C_F^{\chi^{-1}} \neq 0$. 
\end{rem}

\begin{lemma} \label{noexceptionalzeroes}
Let $\theta=\tilde{\chi} \tilde{\omega}^{-1}$. Assume $\tilde{\chi} \tilde{\omega}^{-1}(p) \neq 1$. Then one has $\mu=s_\zeta=0$ for all $\zeta \in \mu_{p^\infty}$ in $\hat{G}^0_{\theta}$.
\end{lemma}
\begin{proof}
 The $\mu$-invariant is zero by \cite{FW79}. 
Let $\zeta \in \mu_{p^\infty}$. We need to show that $\hat{G}_\theta$ does not have a zero at $T=\zeta u^{-1}-1$. We calculate  $\hat{G}_\theta(\zeta u^{-1}-1)=G_{\theta \tilde \omega^2}(u^2 \zeta u^{-1}-1)$. By (1.4)  in \cite{Wiles90} and \eqref{Lpinterpolation}  this equals $L_p(0, \tilde \chi \tilde \omega \chi_\zeta)=L(0, \tilde \chi \chi_\zeta)(1-\tilde \chi \tilde \omega^{-1} \chi_\zeta(p))$. 

Since $\tilde \chi \chi_\zeta$ is odd (as $\chi_\zeta(-1)=+1$ since $-1$ is not congruent to 1 modulo $p$) we have $L(0, \tilde \chi \chi_\zeta) \neq 0$ by the class number formula. Since $\tilde \chi$ is of order prime to $p$ the Euler factor could only vanish for $\zeta=1$ and if $(\tilde \chi \tilde \omega^{-1})(p)=1$.
\end{proof}

Theorem \ref{Wiles2} and  Lemma \ref{noexceptionalzeroes} imply the following corollary (note that by Remark \ref{r4.4} we have $\chi \neq \omega^{\pm 1}$ ).
\begin{cor}\label{Wiles} 
One has $$\bfT/J \cong \Lambda/\hat{G}_{\tilde{\chi} \tilde{\omega}^{-1}}(T).$$
\end{cor}

\subsection{The weight one specialisations}
We set $\bfT_k:= \bfT/\ker \nu_{k,1} \bfT$. It is a well-known result of Hida that for $k\geq 2$ the algebra $\bfT_k$ coincides with the Hecke algebra acting on the space of classical modular forms $S^0_k(Np, \tilde\chi)$, and that all the specialisations are classical, but this is not the case in weight 1. We write $J_{k}$ for the image of $J$ in $\bfT_{k}$ under the map $\bfT\to \bfT/\ker \nu_{k,1}\bfT$ which  we will also denote by $\nu_{k,1}$.

\begin{rem} \label{rem5.5}
A classical specialisation in weight 1  corresponds to Galois representations with finite image, which can only happen if  $\chi=\chi_{F/\bfQ}$ for an imaginary quadratic field $F$, as explained in Remark \ref{CM1}. In the ordinary case such characters are excluded by Remark \ref{exclusions}(ii). In section \ref{sect8} we prove that split deformations of $\rho_0$ with $\chi=\chi_{F/\bfQ}$  are modular by classical CM-forms using a different method.
\end{rem}

 The following lemma will allow us  to later relate $J_1$ to the reducibility ideal.

\begin{lemma} \label{generation of Jk}
The ideal $J_k$ is generated by the set $$S=\{T_{\ell}-1-\tilde\chi(\ell)\tilde{\omega}^{1-k}(\ell)\ell^{k-1}\mid \ell\neq p\} \cup \{T_p-1\}.$$ 
\end{lemma}

\begin{proof} Write $I_k$ for the ideal generated by $S$ and $\fp_k$ for $\ker \nu_{k,1}$. Consider the following commutative diagram
\be \xymatrix{&0\ar[d]&0\ar[d]&0\ar[d]\\
0 \ar[r]& \fp_k\bfT \cap J \ar[r]\ar[d]& J \ar[r]\ar[d] & J_k \ar[r]\ar[d] & 0 \\
0 \ar[r] & \fp_k\bfT \ar[r] & \bfT \ar[r] & \bfT_k\ar[r]& 0}
\ee
Note that the bottom row is exact by definition and all the columns are exact by definition. The top row is exact except possibly at $J$. Clearly, $\fp_k\bfT \cap J \subseteq \ker(J\to J_k)$. But the opposite inclusion is also clear since if $\alpha \in J$ dies in $J_k$, then this just means $\alpha \in \fp_k\bfT$. Hence the top row is also exact and we get $J_k \cong J/(\fp_k\bfT \cap J)$. This quotient is naturally a $\bfT_k$-module and this $\bfT_k$-module structure agrees with the one induced from the $\bfT$-module structure.

If $A$ is a set of generators for $J$ as an ideal of $\bfT$ (i.e., as a $\bfT$-module), then the images under $\bfT\to \bfT_k$ of the elements of $A$ generate $J_k$ as a $\bfT_k$-module. We have $A=\{T_{\ell}-c_{\ell}\mid \ell \in \Spec \bfZ\}$ with $c_p=1$ and $c_{\ell}=1+\tilde{\omega}^{-1}\tilde{\chi}(\ell)\ell(1+T)^{a_{\ell}}$ with $\ell=\tilde{\omega}(\ell)(1+p)^{a_{\ell}}$ if $\ell \neq p$ (see Definition \ref{cell}).

 The lemma follows as we have $$\nu_{k,1}((1+T)^{a_{\ell}})=(1+p)^{(k-2)a_{\ell}}=\ell^{k-2}\tilde{\omega}^{2-k}(\ell).$$
\end{proof}

\begin{cor} \label{surjT1} We have a surjection $$\bfT_1/J_1 \twoheadrightarrow \Oo/L(0, \tilde \chi).$$
\end{cor}

\begin{proof}
We note that for $k=1$ and $\zeta=1$ we get \be\begin{split} \nu_{k, \zeta}\circ \hat{G}_{\theta}(T)=&\nu_{k, \zeta}\circ \hat{G}_{\tilde{\chi} \tilde{\omega}^{-1}}(T)=\nu_{k, \zeta}\circ G_{\tilde{\chi} \tilde{\omega}}(u^2(1+T)-1)\\ =&G_{\tilde{\chi} \tilde{\omega}}(u-1) =L_p(0, \tilde{\chi} \tilde{\omega})=(1-\tilde{\chi} \tilde{\omega}^{-1}(p)) L(0, \tilde{\chi}).\end{split}\ee

By Remark \ref{r4.4} our assumptions from section \ref{The residual representation} imply $\chi \neq \omega$, which implies that $\tilde{\chi}\tilde{\omega}^{-1}(p)\not\equiv 1$ mod $\varpi$ since $\chi$ is a Teichmüller lift of $\ov{\chi}$ (see Remark \ref{invariants1}). 

We thus have the following commutative diagram whose vertical arrows are surjective and whose top row is exact by Corollary \ref{Wiles}. In the top row $\Psi$ is the inclusion map and $\Phi$ is the canonical surjection. The maps in the bottom row are defined in the following way: $\psi$ is the natural injection, and $\phi(t)=\nu_{1,1}(\Phi(\tilde t))$, where $\tilde t$ is any lift of $t$ to $\bfT$.
 This is well-defined as $\Phi(\ker \nu_{1,1} \bfT)=\ker \nu_{1,1} +{\hat{G}_{\theta}}\Lambda$ as $\Phi$ is a $\Lambda$-algebra map.

\be \label{diagram1}\xymatrix{0 \ar[r]& J\ar[r]^{\Psi}\ar[d] &\bfT\ar[r]^{\Phi}\ar[d] & \Lambda/\hat{G}_{\theta} \ar[r]\ar[d]^{\nu_{1,1}} & 0 \\
0 \ar[r] & J_1 \ar[r]^{\psi}& \bfT_1 \ar[r]^{\phi} & \Oo/L_p(0, \tilde{\chi} \tilde{\omega})\ar[r]& 0 }\ee 
Note also that the bottom row is clearly exact except possibly at $\bfT_1$. We do not need exactness at $\bfT_1$, only that $\phi$ factors through $\bfT_1/J_1$, which follows from $\Phi \circ \Psi=0$. 
\end{proof}

Let $\fM$ be the maximal ideal of $\bfT$ containing $J$. We write $\bfT_{\fM}$ for the localisation of $\bfT$ at $\fM$. By a standard argument one can view $\bfT_{\fM}$ as a direct summand of $\bfT$ and so $\bfT_{\fM}$ can be naturally viewed as a subring of $\prod_{\mF\in \mS} L_{\mF}$, where $\mS\subset \mS'$ consists of these representatives whose Fourier coefficients are congruent to $c_{\ell}$ modulo the maximal ideal of $L_{\mF}$ for all primes $\ell$.  Similarly we define $\mN\subset \mN'$ to consist of all the newforms whose Fourier coefficients are congruent to $c_{\ell}$ modulo the maximal ideal of $L_{\mF}$ for all primes $\ell$.

Recall that $\bfT$ and hence also $\bfT_{\fM}$ is a finitely generated  $\Lambda$-module. This implies that $\bfT$ is a semi-local ring hence it is the direct product of its localisations. As $\bfT$ is also a free $\Lambda$-module we get that $\bfT_{\fM}$, being a direct summand of $\bfT$, is a projective, and so also flat, $\Lambda$-module. Since finite flat modules over local rings are free, we conclude that $\bfT_{\fM}$ is free over $\Lambda$.

 We will write $r(\bfT_{\fM})$ for the $\Lambda$-rank of $\bfT_{\fM}$. As already discussed there is a $\Lambda$-algebra  map $\bfT \hookrightarrow \prod_{\mF\in \mS'} \Oo_{L_{\mF}}$, so we can also view $\bfT_{\fM}$ as a $\Lambda$-subalgebra of $\prod_{\mF\in \mS}\Oo_{L_{\mF}}$. To be more precise, for each newform $\mF$ let $\lambda_{\mF}: \bfT \to \Oo_{L_{\mF}}$ be the map sending a Hecke operator to its eigenvalue corresponding to $\mF$. 
 We denote the $\Lambda$-algebra  map $\bfT_{\fM} \hookrightarrow \prod_{\mF\in \mS} \Oo_{L_{\mF}}$ by $\iota$. Note that $\iota = \oplus_{\mF\in \mS} \lambda_{\mF}$.

Let $\varphi: \OL\to \ov{\bfQ}_p$ be an extension of $\nu_{1,1}$. Here $\OL$ is the ring of integers of $L$.
 For each $\mF\in \mS$ we define  $\varphi_{\mF}: \Oo_{L_{\mF}}\to \ov{\bfQ}_p$ to be the restriction of $\varphi$ to $\Oo_{L_{\mF}}$. 
  Then the map $\oplus_{\mF\in \mS}\varphi_{\mF} \circ \iota: \bfT_{\fM} \to \prod_{\mF\in \mS}\varphi_{\mF}(\Oo_{L_{\mF}})$ factors through $\bfT_1$. Even more, it factors through $\bfT_{1,\fm}$ where $\fm$ is the maximal ideal of $\bfT_1$ containing $J_1$. Note that $\nu_{1,1}$ restricts to a surjective map $\nu_{1,1}:\bfT_{\fM} \to \bfT_{1,\fm}=\bfT_{\fM}/(\ker \nu_{1,1}\bfT_{\fM})$.
 In other words there exists an $\Oo$-algebra map $\phi$ such that the following diagram commutes:

\be \label{diagram1} \xymatrix@C+=2cm{\bfT_{\fM}\ar[r]^{\iota}\ar[d]^{\nu_{1,1}} & \prod_{\mF\in \mS}\Oo_{L_{\mF}}\ar[d]^{\oplus_{\mF\in \mS}\varphi_{\mF}}\\ \bfT_{1,\fm} \ar[r]^{\phi} &  \prod_{\mF\in \mS}\varphi_{\mF}(\Oo_{L_{\mF}})}\ee 
Note that $\prod_{\mF\in \mS}\varphi_{\mF}(\Oo_{L_{\mF}})$ is a free $\Oo$-module of finite rank.

\begin{prop}\label{inject3} Assume that  there exists an extension $\varphi: \OL\to \ov{\bfQ}_p$  of $\nu_{1,1}$ such that $\varphi_{\mF}\circ \lambda_{\mF}\neq \varphi_{\mF'}\circ \lambda_{\mF'}$ for all $\mF,\mF'\in \mN$ with $\mF'\neq \mF$. Then  $\phi$ is injective.
\end{prop}

\begin{rem}
We note that the assumption in Proposition \ref{inject3} cannot be weakened. The problem is that two $\Lambda$-adic families may cross at weight one (a phenomenon that does not occur in higher weights) - cf. \cite{DimitrovGhate}. To this end let us illustrate this 
issue with a commutative algebra example. 

Consider the case where $\mN$ consists of only two forms $\mF$ and $\mG$ and suppose for simplicity that $\Oo_{L_{\mF}}=\Oo_{L_{\mG}}=\Lambda$. In particular, this implies that $\mF$ is not a Galois conjugate of $\mG$, so $\mS=\mN$ and $\varphi_{\mF}=\varphi_{\mG}=\varphi$.  Then $\iota= \lambda_{\mF}\oplus \lambda_{\mG}$. Assume that the families $\mF$, $\mG$ cross at weight one, i.e.,  that $\varphi\circ \lambda_{\mF}=\varphi\circ\lambda_{\mG}$.

Then while the image of $\varphi\oplus \varphi$ is clearly $\Oo \times \Oo$, the image of $(\varphi\oplus \varphi)\circ \iota$ equals the diagonally embedded copy of $\Oo$ inside $\Oo\times \Oo$. However, the $\Oo$-rank of $\bfT_{1, \fm}$ is still 2 (see proof of Lemma \ref{inject} below), so the map $\phi$ cannot be injective.

It is perhaps worth noting that this potential mismatch between the rank of $\bfT_{1,\fm}$ and the rank of the image of $(\varphi\oplus \varphi)\circ \iota$ means that in general $\bfT_{1, \fm}$ does not act faithfully on weight one specialisations (even the non-classical ones) of $\Lambda$-adic newforms which are congruent to an Eisenstein series. Our method of proving an $R=T$ theorem works when this action is faithful.

More precisely,  we require $\phi$ to be injective because when we construct the map $R\to \bfT_{1,\fm}$ we glue together maps $R\to \Oo$ (cf. Proposition \ref{RtoT}). If $\phi$ was not injective this process would only give a map $R \to \phi(\bfT_{1,\fm})$, which a priori has no reason to factor through $\bfT_{1,\fm}$. One may hope that one could in general prove an isomorphism between $R$ and $\phi(\bfT_{1,\fm})$ instead of $\bfT_{1,\fm}$. Our method however proceeds by showing that $L(0,\tilde \chi)$ is an upper bound on $R/I$ and a lower bound on the Hecke congruence module. While  Corollary \ref{surjT1} can be used to get  a lower bound for $\bfT_{1,\fm}/J_{1,\fm}$ by $L(0, \tilde \chi)$, such a bound would in general not imply a corresponding bound on $\phi(\bfT_{1,\fm})/\phi(J_{1,\fm})$ if $\phi$ were not injective. 

Finally, let us record here that the injectivity of $\phi$ implies that $\bfT_{1,\fm}$ is reduced, even though we do not use this last fact directly.
\end{rem}

\begin{rem} We note that in the presence of Galois conjugate forms,  it is not enough to assume that  $\varphi_{\mF}\circ \lambda_{\mF}\neq \varphi_{\mF'}\circ \lambda_{\mF'}$ for all $\mF,\mF'\in \mS$ with $\mF'\neq \mF$. This is so because it is a priori possible for two elements of a $\Lambda$-adic  Galois conjugacy class to specialize to the same weight one form. \end{rem}

\begin{proof}[Proof of Proposition \ref{inject3}]
We begin by proving two lemmas.
\begin{lemma}\label{inject} The $\Oo$-module $\bfT_{1,\fm}$ is finitely generated and free of rank $r(\bfT_{\fM})$. If the $\Oo$-rank of $\phi(\bfT_{1,\fm})$ equals $r(\bfT_{\fM})$, then  $\phi$ is injective.
\end{lemma}

\begin{proof}   One has $\bfT_{1,\fm} = \bfT_{\fM}\otimes_{\Lambda} \Lambda/\ker \nu_{1,1}$. As $\Lambda/\ker_{\nu_{1,1}}\cong \Oo$ and $\bfT_{\fM}\cong \Lambda^{r(\bfT_{\fM})}$  as $\Lambda$-modules, we get that $\bfT_{1,\fm}$ is a finitely generated free $\Oo$-module of rank $r(\bfT_{\fM})$. Finally, if the $\Oo$-rank of $\phi(\bfT_{1,\fm})$ equals $r(\bfT_{\fM})$ we conclude that $\ker \phi$ is a torsion submodule of $\bfT_{1,\fm}$, so must be zero. 
\end{proof}

\begin{lemma}\label{inject2}  If the $\Oo$-rank of  the image of  $ \oplus_{\mF\in \mS}\varphi_{\mF} \circ \iota$ equals the $\Lambda$-rank of $\prod_{\mF\in \mS}\Oo_{L_{\mF}}$, then  $\phi$ is injective.
\end{lemma}

\begin{proof}  First note that as $\bfT_{\fM}\otimes_{\Lambda}F_{\Lambda} \cong \prod_{\mF\in \mS}L_{\mF}$ we get that the image of the embedding $\iota$ is a $\Lambda$-submodule of $\prod_{\mF\in \mS}\Oo_{L_{\mF}}$ of full rank. So, we conclude that the $\Lambda$-rank of $\prod_{\mF\in \mS}\Oo_{L_{\mF}}$ equals $r(\bfT_{\fM})$. The corollary then follows from the commutativity of \eqref{diagram1} and Lemma \ref{inject}. \end{proof}

 In light of Lemma \ref{inject2} it is enough to prove that the $\Oo$-rank $s$ of  the image $I$ of  $ \oplus_{\mF\in \mS}\varphi_{\mF} \circ \iota$ equals the $\Lambda$-rank of $\prod_{\mF\in \mS}\Oo_{L_{\mF}}$.
From its proof we also see that the $\Lambda$-rank  of $\prod_{\mF\in \mS}\Oo_{L_{\mF}}$ equals $r(\bfT_{\fM})$. So, we need to show that $s=r(\bfT_{\fM})$.

As the map $\oplus_{\mF\in \mS}\varphi_{\mF}$ is surjective, we get that $s\leq r(\bfT_{\fM})$. For the reverse inequality first note that $r(\bfT_{\fM})=\#\mN$. Indeed, we have that
$\bfT\otimes_{\Lambda}L\cong \prod_{\mF\in \mN'} L$ by \cite{Wiles90}, p. 507 from which it follows that $\bfT_{\fM} \otimes_{\Lambda} L \cong \prod_{\mF\in \mN} L$.

 Thus it suffices to prove that $s\geq \#\mN$. Note that the map $$f\mapsto (i\otimes 1 \mapsto f(i))$$ gives rise to an injective map $$ \Hom_{\Oo-{\rm alg}}(I, \ov{\bfQ}_p)\hookrightarrow \Hom_{\ov{\bfQ}_p-{\rm alg}}(I\otimes_{\Oo}\ov{\bfQ}_p, \ov{\bfQ}_p).$$ As the $\Oo$-rank of $I$ equals $s$, we get that $\dim_{\ov{\bfQ}_p}(I\otimes_{\Oo}\ov{\bfQ}_p)=s$ and as $I\otimes_{\Oo}\ov{\bfQ}_p$ is an Artinian ring, it is a product of fields, so we must have $$I\otimes_{\Oo}\ov{\bfQ}_p\cong \ov{\bfQ}_p^s \quad \textup{as $\ov{\bfQ}_p$-algebras.}$$ As the only $\ov{\bfQ}_p$-algebra homomorphisms from $\ov{\bfQ}_p^s$ to $\ov{\bfQ}_p$ are projections, we get that $\# \Hom_{\ov{\bfQ}_p-{\rm alg}}(I\otimes_{\Oo}\ov{\bfQ}_p, \ov{\bfQ}_p)=s$. It follows that $\#  \Hom_{\Oo-{\rm alg}}(I, \ov{\bfQ}_p)\leq s$. Thus by Lemma \ref{HomsfromI} below we get $s\geq \#\mN$, as desired. \end{proof}

\begin{lemma} \label{HomsfromI}
Under the assumptions of Proposition \ref{inject3} one has $$\# \Hom_{\Oo-{\rm alg}}(I, \ov{\bfQ}_p)\geq \#\mN.$$
\end{lemma}

\begin{proof}
By our non-crossing assumption we know that we have $\#\mN$ distinct maps $\varphi_{\mF}\circ\lambda_{\mF}: \bfT_{\fM} \to \ov{\bfQ}_p$. It suffices to show that each of them factors through $I$. For each $\mF\in \mS$ this follows from the commutativity of the following diagram:

\be \label{diag3}\xymatrix@C+=2cm{\bfT_{\fM} \ar[r]^{\iota}\ar[dr]_{\lambda_{\mF}}& \prod_{\mF\in \mS}\Oo_{L_{\mF}} \ar[r]^{\oplus_{\mF\in \mS} \varphi_{\mF}} \ar[d]^{\pi_{\mF}} & \prod_{\mF\in \mS}\varphi_{\mF}(\Oo_{L_{\mF}}) \ar[d]^{\pi_{\mF}}\\ & \Oo_{L_{\mF}} \ar[r]^{\varphi_{\mF}}&\varphi(\Oo_{L_{\mF}})}\ee (the triangle commutes by the definition of $\lambda_{\mF}$ and the square commutes also by the definition of the maps involved).

Now, let $\mF'$ be a Galois conjugate of $\mF$. Then  there exists a Galois element $\sigma=\sigma(\mF')\in \Gal(L/F_{\Lambda})$ such that $\mF'=\sigma \mF$, and so $\Oo_{L_{\mF'}}=\sigma \Oo_{L_{\mF}}$. Note that $\lambda_{\mF'} = \sigma \circ \lambda_{\mF}$.

With this we amend the diagram \eqref{diag3} and obtain a new commutative diagram. 

\be \label{diag4}\xymatrix@C+=2cm{\bfT_{\fM} \ar@/_/[ddr]_{\lambda_{\mF'}}\ar[r]^{\iota}\ar[dr]_{\lambda_{\mF}}& \prod_{\mF\in \mS}\Oo_{L_{\mF}} \ar[r]^{\oplus_{\mF\in \mS} \varphi_{\mF}} \ar[d]^{\pi_{\mF}} & \prod_{\mF\in \mS}\varphi_{\mF}(\Oo_{L_{\mF}}) \ar[d]^{\pi_{\mF}}\\ & \Oo_{L_{\mF}} \ar[r]^{\varphi_{\mF}}\ar[d]^{\sigma}&\varphi(\Oo_{L_{\mF}})\ar@{^{(}->}[r]&\ov{\bfQ}_p\ar@{=}[d]\\
&\Oo_{L_{\mF'}}\ar[r]_{\varphi_{\mF'}}&\varphi(\Oo_{L_{\mF'}})\ar@{^{(}->}[r]&\ov{\bfQ}_p
}
\ee

From this diagram we see that $\varphi_{\mF'}\circ \lambda_{\mF'}: \bfT_{\fM} \to \ov{\bfQ}_p$ also factors through $I$, as desired.
\end{proof}

\subsection{Galois representations in weight one}
For each $\Lambda$-adic newform $\mF$ we have an associated Galois representation.

\begin{thm}[Hida, Wiles, Carayol]\label{Carayol}
Let $\mathcal{F} \in S^0_{\Oo_L}(N, \theta)$ be a newform (recall that $\theta$ is assumed to be primitive).
Then there exists a continuous irreducible odd Galois representation $$\rho_\mathcal{F}: G_{\bfQ} \to {\rm GL}_2(L)$$  unramified outside $Np$ such that $${\rm Tr}(\rho_\mathcal{F}(\Frob_\ell))=c(\ell, \mathcal{F})$$ for all primes $\ell \nmid Np$ and $$\det(\rho_\mathcal{F}(\Frob_\ell))=\theta(\ell) \ell (1+T)^{a_\ell},$$ where $\ell=\tilde \omega(\ell) (1+p)^{a_\ell}$ for $a_\ell \in \bfZ_p$.
\begin{enumerate} \item We have $\rho_\mathcal{F}|_{D_p} \cong \begin{pmatrix} \epsilon_1&*\\0& \epsilon_2 \end{pmatrix}$ with $\epsilon_2$ unramified and $\epsilon_2(\Frob_p)=c(p, \mathcal{F})$.
\item For  $\ell \mid N$ we have $\rho_\lambda|_{D_\ell}=\begin{pmatrix} \psi &0\\0& \delta_\ell \end{pmatrix}$ with $\delta_\ell$ unramified and $\delta_\ell(\Frob_\ell)=c(\ell, \mathcal{F})$.
\end{enumerate}
\end{thm}

\begin{definition} \label{definitionof rho1F} We will write  $\rho_{\mF}^1:G_{\Sigma} \to \GL_2(\ov{\bfQ}_p)$ for the semi-simple Galois representation associated with $\varphi_{\mF}\circ \tr \rho_{\mF}$. \end{definition} Note that $\det \rho_{\mF}^1=\chi$ using that $\theta =\tilde{\chi}\tilde{\omega}^{-1}$ as in the proof of Lemma \ref{generation of Jk}. 
  Recall that  $\mS\subset \mS'$ consists of $\mF$ with $\varphi_{\mF}(\mF)$ whose Hecke eigenvalue at $\ell$ is congruent to $1+\tilde{\chi}(\ell)$ mod $\varpi$ for all primes $\ell \neq p$ and $\varphi_{\mF}(c(p,\mF))\equiv 1 \mod{\varpi}$. 
Recall that we assume  that $C_F^{\chi^{-1}}\otimes_{\Oo}\bfF$ has dimension one. As $\#C_F^{\chi^{-1}}= \#\Oo/L(0, \chi)$ (see \eqref{size of class group}) we conclude that $\val_{\varpi}(L(0, \chi))>0$. By Corollary \ref{surjT1} we get that $J_1 \neq \bfT_1$, so $\mS$ is not empty. For $\mF \in \mS$ the semi-simplification of the mod $\varpi$ reduction $\ov{\rho}_{\mF}^1$ of $\rho_{\mF}^1$ has the form $1\oplus \ov{\chi}$.

\begin{thm} \label{irreducibility} For $\mF \in \mS$ the representation $\rho^1_{\mF}: G_{\Sigma} \to \GL_2(\ov{\bfQ}_p)$ is irreducible.
\end{thm}
\begin{proof}
Suppose $\tr \rho^1_{\mF}$ is a sum of two characters $\psi_1$   and $\psi_2$ such that $\psi_1$ reduces to 1 mod $\varpi$ and $\psi_2$ reduces to $\ov{\chi}$. By ordinarity we can assume   that $\psi_1$ is unramified at $p$. Furthermore, by  Theorem \ref{Carayol} (2) and the fact that $\det \rho_{\mF}^1=\chi$ we see that for $\ell \in \Sigma-\{ p\}$ we have $$(\rho^1_{\mF})^{\rm ss}|_{I_{\ell}} =1\oplus \chi|_{I_{\ell}}.$$
By Remark \ref{invariants1} this forces $\psi_1=1$ and thus $\psi_2=\chi$.

Let $k$ be a positive integer such that $k \equiv 1$ mod $(p-1)$. 
 Let $\varphi_{\mF,k}$ be an extension of $\nu_{k,1}$. Let $E_k$ be a finite extension of the compositum of $\varphi_{\mF,k}(L)\varphi_{\mF}(L)$ and write $\Oo_k$ for its ring of integers. This is a finite extension of $\bfQ_p$. Let $\varpi_k$ be a uniformizer of $\Oo_k$.  As $\mF$ is a cusp form we know that $\varphi_{\mF,k}(\mF)$ is a cusp form for an infinite subset $\mW$ of $k$s as above. Assume that $k\in \mW$. 
Similarly to the case of weight 1,  composing $\rho_{\mF}$ with $\varphi_{\mF,k}$ gives rise to a Galois representation $\rho^k_{\mF}: G_{\Sigma}\to \GL_2(\varphi_{\mF,k}(L))$. By cuspidality of $\varphi_{\mF,k}(\mF)$, this Galois representation is irreducible. 
Then by Proposition 2.1 in \cite{Ribet76} there exists a lattice inside the representation space of $\rho^k_{\mF}$ such that with respect to that lattice the mod $\varpi_k$ reduction $\ov{\rho}^k_{\mF}$ of $\rho_{\mF}^k$ is of the form \be \label{cong11} \ov{\rho}^k_{\mF} = \bmat 1 &*\\ &\ov{\chi}\emat \not\cong 1 \oplus \ov{\chi}. \ee

Let $m_k$ be the largest positive integer $m$ such that $\varphi_{\mF}(c(1,t\mF))\equiv \varphi_{\mF,k}(c(1,t\mF))$ mod $\varpi^m$ for all $t \in \bfT$. Note that this makes sense as $\varpi\in \Oo_k$. 
Using congruence \eqref{cong11} and Theorem in \cite{Urban01} we conclude that there is a lattice $\Lambda$ in the space of $\rho^k_{\mF}$ such that with respect to a certain basis $\{e_1, \e_2\}$ one has \be \label{modvarpimk}\rho^k_{\mF} = \bmat 1 &*\\ &\chi\emat \pmod{\varpi^{m_k}} \ee with $*$ still not split mod $\varpi_k$.

 So by Theorem \ref{Carayol}(1) we get $$\rho^k_{\mF}|_{D_p}\cong_{E_k}\bmat \chi \beta^{-1} \omega^{1-k} \epsilon^{k-1}&* \\ &\beta\emat,$$ where $\beta$ is unramified and maps $\Frob_p$ to  $\varphi_{\mF,k}(c(p,\mF))$. 

We claim that it is possible to change the basis of $\Lambda$ such that in that new basis $$\rho^k_{\mF}|_{D_p}=\bmat \chi \beta^{-1} \omega^{1-k} \epsilon^{k-1}&* \\ &\beta\emat.$$

By ordinarity
there exists a vector $v=ae_1 + be_2$, with $a, b \in E_k$, on which $D_p$ acts
by $\chi \beta^{-1} \omega^{1-k} \epsilon^{k-1}$. Multiply this by a power of $\varpi_k$ such that $\varpi_k^s a, \varpi_k^s b \in \Oo_k$ and $v':= \varpi_k^s v
\not \equiv 0 \mod \varpi_k$. 

Assume that $\varpi_k^s a$  is a
$\varpi_k$-unit. Then we have $\Lambda=\Oo_k v' + \Oo_k e_2$. (If $\varpi_k^s b \in \Oo_k^\times$
then $\Lambda=\Oo_k v' + \Oo_k e_1$.)
As $\det(\rho^k_{\mF})=\chi  \omega^{1-k} \epsilon^{k-1}$ we see that in the basis $\mB'=\{v', e_2\}$ (respectively $\mB'=\{v', e_1\}$) we have \be \label{ordrhofk} \rho^k_{\mF}|_{D_p}=\bmat \chi \beta^{-1} \omega^{1-k} \epsilon^{k-1}&* \\ &\beta\emat. \ee

We know that $\beta \equiv 1 \mod{\varpi_k}$ as $\mF \in \mS$ and $\chi|_{D_p} \not \equiv 1 \mod{\varpi_k}$ by assumption, hence $\beta \not \equiv \chi|_{D_p} \mod{\varpi_k}$. So Lemma \ref{uniqueness of Ti} tells us that $\beta\equiv 1$ mod $\varpi^{m_k}$ by comparing the above to \eqref{modvarpimk}. 

Thus we get that in the basis $\mB'$ of $\Lambda$ we have $$\rho_{\mF}^k|_{D_p} = \bmat \chi & *'\\ & 1\emat \pmod{\varpi^{m_k}=\varpi_k^{e_km_k}},$$ where $e_k$ is the ramification index of $\Oo_k$ over $\Oo$.  Comparing this with \eqref{modvarpimk} we conclude that there exists a matrix $\bmat A&B\\C&D\emat \in \GL_2(\Oo_k/\varpi^{m_k})$ such that \be \label{matrix9} \bmat A&B\\C&D\emat \bmat 1&*\\ & \chi\emat = \bmat \chi &*'\\ & 1 \emat \bmat A&B\\C&D\emat\ee where $\chi$ and $*$ are considered after restrictions to $D_p$.  We first note that $\ov{C}\neq 0$ mod $\varpi_k$. Indeed, if the reduction $\ov{C}$ of $C$ mod $\varpi_k$ were 0 then reducing the equation \eqref{matrix9} mod $\varpi_k$ and comparing the top-left entries we would get $\ov{A}=\ov{\chi}\ov{A}$. Note that if $\ov{C}=0$ then we must have $\ov{A}\neq 0$ as the matrix $\bmat A&B\\ C&D\emat$ is invertible. This contradicts the assumption that $\ov{\chi}|_{D_p}\neq 1$. Hence we must have that $C$ is a unit in $\Oo_k/\varpi^{m_k}$.

Now compare the top-left entries of \eqref{matrix9} to get that $A=A\chi + *'C$, from which we get that $*'=(A/C)(1-\chi)$. Hence the cocycle induced by $*'$ in $H^1(D_p, \Oo/\varpi^{m_k}(\chi))$  is a coboundary. In other words, $$\rho_{\mF}^k|_{D_p} \cong \chi \oplus 1 \pmod{\varpi^{m_k}}$$

Hence the $*$ in \eqref{modvarpimk}  gives rise to an element  $c_k\in H^1_{\Sigma-\{p\}}(\bfQ, \Oo_k/\varpi^{m_k}(\chi^{-1}))$ which is not annihilated by $\varpi^{m_k-1}$.
 We now use Lemma \ref{lower bound 21} for primes $\ell \mid N$ or such that $\tilde\chi(\ell)\ell \not \equiv 1 $ mod $p$ to deduce that $c_k \in H^1_{\emptyset}(\bfQ, \Oo_k/\varpi^{m_k}(\chi^{-1}))$ using the fact that by  assumption (2) in section \ref{The residual representation} this covers all the primes in $\Sigma-\{p\}$.

We now claim that there exists an element of $H^1_{\emptyset}(\bfQ,E/\Oo(\chi^{-1}))$ which is not annihilated by $\varpi^{m_k-1}$. For this first note that for every positive integer $r$ one has  $\Oo_k/\varpi^r=(\Oo/\varpi^r)^s$ where $s=[E_k:E]$.  As the formation of Selmer groups commutes with direct sums we get $$H^1_{\rm f}(\bfQ, \Oo_k/\varpi^{m_k})\cong (H^1_{\rm f}(\bfQ, \Oo/\varpi^{m_k}(\chi^{-1})))^s.$$ 
Since $\varpi^{m_k-1}c_k \neq 0$ we conclude that there must exist an element $c'_k\in H^1_{\emptyset}(\bfQ ,\Oo/\varpi^{m_k}(\chi^{-1}))$ such that $\varpi^{m_k-1}c'_k\neq 0$.

By \eqref{functoriality} we have that $\iota: H^1_{\emptyset}(\bfQ,\Oo/\varpi^r(\chi^{-1}))\rightarrow H^1_{\emptyset}(\bfQ, E/\Oo(\chi^{-1}))[\varpi^r]$  is an isomorphism.
Therefore the elements $c_k$ give rise to an infinite sequence of elements  $c'_k\in H^1_{\emptyset}(\bfQ, E/\Oo(\chi^{-1}))$ for $k \in \mW$ with the property that $\varpi^{m_k-1}c'_k\neq 0$. 

As $m_k \to \infty$  when $k$ approaches 1 $p$-adically this forces $H^1_{\emptyset}(\bfQ, E/\Oo(\chi^{-1}))$ to be infinite. However, one has by Proposition \ref{clgroup} that  $$H^1_{\emptyset}(\bfQ, E/\Oo(\chi^{-1}))\cong \Hom(\Cl(\bfQ(\chi)),E/\Oo(\chi^{-1}))^{\Gal(\bfQ(\chi)/\bfQ)},$$ so we get a contradiction to the finiteness of class groups.
\end{proof}

\subsection{Modularity of reducible deformations} From now on we will assume the non-crossing assumption of Proposition \ref{inject3}, i.e., that there exists an extension $\varphi: \OL\to \ov{\bfQ}_p$ of $\nu_{1,1}$ 
such that
$\varphi_{\mF}\circ \lambda_{\mF}\neq \varphi_{\mF'}\circ \lambda_{\mF'}$ for all $\mF,\mF'\in \mN$ with $\mF'\neq \mF$.
Then by Proposition \ref{inject3} we can identify $\bfT_{1,\fm}$ with its image inside $\prod_{\mF\in \mS} \varphi_{\mF}(\Oo_{L_{\mF}})$ under the map $\phi$. 
 For every $\mF\in \mS$, we have that $\rho^1_{\mF}: G_{\Sigma} \to \GL_2(\varphi_{\mF}(L_{\mF}))$ is irreducible by Theorem \ref{irreducibility}. 
 Then by Proposition 2.1 in \cite{Ribet76} there exists a $G_{\Sigma}$-stable lattice $\Lambda$ in the space of $\rho_{\mF, \Lambda}^1$ such that with respect to that lattice we have 

\be \label{Ribet1} \ov{\rho}_{\mF, \Lambda}^1 = \bmat 1 & * \\ & \ov{\chi}\emat \not\cong 1\oplus \ov{\chi}.\ee

Furthermore by (1) in Theorem \ref{Carayol} we know that $\rho^1_{\mF}|_{D_p}$ has a $D_p$-stable $\varphi_{\mF}(L_{\mF})$-line $L$ on which $D_p$ acts via $\chi \beta^{-1}$ and a quotient on which $D_p$ acts by $\beta$ with $\beta$ an unramified $\varphi_{\mF}(\Oo_{L_{\mF}})$-valued character which reduces to the identity  mod $\varpi$ (because $\mF\in \mS$ and $\ov{\chi}|_{D_p}\neq 1$ by our assumption). 

Arguing as in the proof of Theorem \ref{irreducibility} we conclude that this combined with  \eqref{Ribet1} shows that $\ov{\rho}^1_{\mF, \Lambda}$ splits when restricted to $D_p$. 
In particular, it splits when restricted to $I_p$. 
Hence it follows from Proposition \ref{uniqueness} that for any $\mF, \mF'$ as above we can choose  lattices $\Lambda_{\mF}, \Lambda_{\mF'}$ such that $\ov{\rho}^1_{\mF, \Lambda_{\mF}}= \ov{\rho}^1_{\mF', \Lambda_{\mF'}}$ (note that $\mS\neq \emptyset$). For each $\mF \in \mS$ we make such a choice and set $$\rho_0 = \ov{\rho}^1_{\mF, \Lambda_{\mF}}.$$

For any finite set of primes $\Sigma'$ we define $\bfT^{\Sigma'}$ to be the $\Oo$-subalgebra of $\bfT_{1,\fm}$ generated by the images under the map $\bfT\twoheadrightarrow \bfT_1
\to \bfT_{1, \fm}$ of the operators $T_p$ and $T_{\ell}$ for all primes $\ell \not\in \Sigma'$. In particular, for $\Sigma'\subset \Sigma''$ there is a natural $\Oo$-algebra map $\bfT^{\Sigma''} \hookrightarrow \bfT^{\Sigma'}$ and $\bfT^{\emptyset}=\bfT_{1, \fm}$. Define $J^{\Sigma'}$ as the ideal of  $\bfT^{\Sigma'}$ generated by the image of the set $S^{\Sigma'}=\{T_{\ell}-1-\tilde{\chi}(\ell)\mid \ell \not\in \Sigma', \ell \neq p\}\cup \{T_p-1\}$. Note that $S^{\emptyset}$ is the same as the set $S$ in Lemma \ref{generation of Jk}, hence $J^{\emptyset}=J_{1, \fm}$. The injection $\bfT^{\Sigma''} \hookrightarrow \bfT^{\Sigma'}$ induces an $\Oo$-algebra map $\bfT^{\Sigma''}/J^{\Sigma''} \to \bfT^{\Sigma'}/J^{\Sigma'}$ which is surjective as the structure map $\Oo \twoheadrightarrow \bfT^{\Sigma'}/J^{\Sigma'}$ is surjective by the definition of $S^{\Sigma'}$. Hence, in particular, we have an $\Oo$-algebra surjection \be\label{Osurj} \bfT^{\Sigma}/J^{\Sigma} \twoheadrightarrow \bfT_{1, \fm}/J_{1, \fm},\ee
 where we recall that $\Sigma$ is our fixed finite set of primes containing $p$ and primes dividing $N$.

The following proposition gives us that every reducible deformation  (i.e., every deformation that factors through $R/I$) is modular. 
\begin{prop} \label{RtoT} There exists a surjective $\Oo$-algebra map $\Phi: R \to \bfT^\Sigma$ given by $\tr \rho^{\rm univ}(\Frob_{\ell}) \mapsto (\varphi_\mF(c(\ell, \mF)))_{\mF \in \mS} $ for all $\ell \not\in \Sigma$ (cf. Proposition \ref{genbytraces} that this indeed defines a map) which induces an isomorphism $R/I \xrightarrow{\sim} \bfT^{\Sigma}/J^{\Sigma}$.
\end{prop}

\begin{proof}
Let $\mF \in \mS$. By the discussion above we know that there exists a lattice $\Lambda$ such that $\ov{\rho}^1_{\mF, \Lambda}$
 is equal to $\rho_0$.

By ordinarity of $\mF$ (see Theorem \ref{Carayol}(1)) we know that $\rho_{\mF, \Lambda}^1|_{D_p}\cong \bmat \chi \phi_{c_p}^{-1}&*\\ &  \phi_{a_p}\emat$ for $a_p=\varphi_{\mF}(c(p, \mF))$. Recall that $a_p \equiv 1 \mod{\varpi}$.  Since $N$ is the conductor of $\tilde\chi$ and the tame level of $\mF$  we see that by Theorem \ref{Carayol}(2) the condition (iii) of our deformation conditions is satisfied. Indeed, if $\ell \mid N$ we have $\rho^1_{\mF, \Lambda}|_{I_{\ell}}=\bmat \psi \\ & 1\emat$, so $\psi=\chi$ (as $\det \rho_{\mF, \Lambda}=\chi$). If $\ell \in \Sigma$, but $\ell \nmid N$, then $\rho_{\mF,\Lambda}^1$ is unramified at $\ell$. Hence $\rho_{\mF, \Lambda}^1$ is a deformation of $\rho_0$. 

So, we get an $\Oo$-algebra map $\Phi: R \to \prod_{\mF \in \mS} \varphi_\mF(\Oo_{L_\mF})$. We claim that ${\rm im}(\Phi) \supset \bfT^\Sigma$. 
Clearly all operators $T_{\ell}$ for primes $\ell \nmid Np$ are in the image of $\Phi$.
For $T_p$ we adapt an argument from the proof of \cite{WakeWangErickson21} Proposition A.2.3:  Since we assume that $\overline \chi|_{D_p} \neq 1$ there exists $\sigma \in D_p$ lifting $\Frob_p \in D_p/I_p\cong G_{\bfF_p}$ such that $\overline \chi(\sigma) \neq 1$. It is clear that it can be done if $\overline{\chi}|_{I_p} =1$. Otherwise take $\sigma'$, any lift of Frobenius. If it happens to satisfy $\ov{\chi}(\sigma')=1$, multiply $\sigma'$ by an element in inertia for which $\overline \chi$ is non-trivial.

By  deformation condition (ii) the characteristic polynomial for $\rho^{\rm univ}(\sigma)$  is $$x^2-(\psi_1(\sigma)+\psi_2(\sigma))x + \chi(\sigma)$$
and this reduces modulo $\fm_R$ to $$(x-1)(x-\overline \chi(\sigma)).$$

Let $U \in R$ be the root of the characteristic polynomial of $\rho^{\rm univ}(\sigma)$ such that $U \equiv 1 \mod{\fm_R}$ (which exists and is unique by Hensel's lemma). We claim that $\Phi(U)=T_p$.  It suffices to check for each $\mF \in \mS$ that $\Phi(U)$ projects to $\varphi_{\mF}(c(p, \mF))$ in $\varphi_\mF(\Oo_{L_\mF})$. Fix such an $\mF$ and write $a_p:=\varphi_{\mF}(c(p, \mF))$. Since $\rho_\mF^1$ is a deformation of $\rho_0$ we know that $U$ maps to a root of the characteristic polynomial of $\rho_\mF^1(\sigma)$. Since $\sigma$ is a lift of $\Frob_p$ we know by Theorem \ref{Carayol}(1) that this characteristic polynomial equals $$(x-a_p)(x-\chi(\sigma)a_p^{-1}),$$ so $U$ must map to $a_p \equiv 1 \mod{p}$, as required.

 The arguments above prove that ${\rm im}(\Phi) \supset \bfT^\Sigma$.  On the other hand, Proposition \ref{genbytraces} implies that this image is contained in $\bfT^{\Sigma}$, so we conclude that it equals $\bfT^\Sigma$. Composing the map $\Phi: R\twoheadrightarrow \bfT^{\Sigma}$ with $\rho^{\rm univ}$ we get a representation into $\GL_2(\bfT^{\Sigma})$ whose trace is reducible modulo $J^{\Sigma}$. So we get 
$\Phi(I) \subseteq J^{\Sigma}$. Hence using Corollary \ref{surjT1} and \eqref{Osurj} we obtain a surjection $$R/I \twoheadrightarrow \bfT^\Sigma/J^{\Sigma} \twoheadrightarrow \bfT_{1,\fm}/J_{1,\fm} \cong  \bfT_1/J_1 \twoheadrightarrow   \Oo/L(0, \chi).$$

By Proposition \ref{bound on R/I} we get that $\# R/I \leq \# C_F^{\chi^{-1}}$. The Proposition now follows from the fact that $\# C_F^{\chi^{-1}}=\# \Oo/L(0, \chi)$ (see \eqref{size of class group}).
\end{proof}

\subsection{The main result} \label{summary} For the reader's convenience we repeat here all our assumptions and also indicate how they are used.

Let $p>2$ be a prime and $N$ a positive integer with $p\nmid N$. Let $\tilde\chi: (\bfZ/Np\bfZ)^{\times} \to \bfC^{\times}$ 
denote a Dirichlet character of order prime to $p$ [this corresponds to type S characters considered in \cite{Wiles90} and so is used in Theorem \ref{Wiles2}; it is also used in
 Proposition \ref{clgroup} and Theorem \ref{principality}] with $\tilde{\chi}(-1)=-1$. Write $\tilde\chi=\tilde\chi_p\tilde\chi_N$ where $\tilde\chi_p$ is a Dirichlet character mod $p$ (i.e., a character of $(\bfZ/p\bfZ)^{\times}$) and $\tilde\chi_N$ is  a character mod $N$. Assume $\tilde\chi_N$ is primitive [as required by \cite{Wiles90}].

Let $\Sigma$ be a finite set of primes containing $p$ and the primes dividing $N$. 
If $\ell\in \Sigma$ is a prime such that $\ell \nmid Np$ then we require that $\ell$ satisfies: 
\begin{enumerate}
\setcounter{enumi}{1}
\item  $\tilde\chi(\ell)\ell \not \equiv 1$ mod $\varpi$;
\item $\tilde\chi(\ell) \not \equiv \ell$ mod $\varpi$.
\end{enumerate}
[Assumption (2) comes in for Propositions \ref{uniqueness}, \ref{bound on R/I}, Theorem \ref{principality} and \ref{irreducibility} via Lemma \ref{lower bound 21}, while assumption (3) is only used for Theorem \ref{principality}).

Write $\chi: G_{\Sigma} \to \Oo^{\times}$ for the Galois character associated to $\tilde\chi$ and $\ov{\chi}: G_{\Sigma} \to \bfF^{\times}$ for its mod $\varpi$ reduction. 
We assume that $\ov{\chi}|_{D_p} \neq 1$.
[This is used e.g. for Propositions \ref {infi}, \ref{bound on R/I} and Theorem \ref{irreducibility}, and to ensure that the Galois representations associated to an ordinary eigenform have a lattice reducing to a $\rho_0$ split at $I_p$.]

Write $F:= \bfQ(\chi)$ for the splitting field of $\chi$ and ${\rm Cl}(F)$ for the class group of $F$. Set $C_F := {\rm Cl}(F)\otimes_{\bfZ} \Oo$. For any character $\psi: \Gal (F/\bfQ) \to \Oo^{\times}$ we write $C_F^{\psi}$ for the $\psi$-eigenspace of $C_F$ under the canonical action of $\Gal(F/\bfQ)$.

Assume that $C_F^{\chi^{-1}}\neq 0$ (which is equivalent to assuming that $\val_p(L(0,\chi))>0$). Consider a continuous homomorphism 
$\rho_0: G_{\Sigma} \to \GL_2(\bfF)$ of the form $$\rho_0 = \bmat 1 & * \\ 0 & \ov{\chi} \emat \not\cong 1 \oplus \ov{\chi}$$ such that $\rho_0|_{D_p}\cong 1\oplus  \ov{\chi}|_{D_p}$. In fact our assumptions force the existence of such a $\rho_0$ (see \eqref{Ribet1} and the discussion  following it).

Put $\bfT_1:= \bfT/\ker \nu_{1,1} \bfT$ to be the weight 1 specialisation of $\bfT$, the cuspidal 
$\Lambda$-adic ordinary Hecke algebra of tame level $N$. Write $\fM \subset \bfT$ for the maximal ideal containing the Eisenstein ideal $J$ and  $\fm \subset \bfT_1$ for the maximal ideal containing its Eisenstein ideal  i.e., the ideal of $\bfT_1$ generated by $\{T_\ell-1-\tilde \chi(\ell) | \ell \neq p\} \cup \{T_p-1\}$.  We define $\bfT^{\Sigma}$ to be the $\Oo$-subalgebra of  the localisation  $\bfT_{1,\fm}$  of $\bfT_1$  at the ideal $\fm$  generated by the images under the map $\bfT\twoheadrightarrow \bfT_1
\to \bfT_{1, \fm}$ of the operators $T_p$ and $T_{\ell}$ for all primes $\ell \not\in \Sigma$.  

Let $L\subset \ov{\bfQ}_p$ be a finite extension of $\bfQ_p$ which contains the values of all $\Oo$-algebra homomorphisms $\lambda_{\mF}: \bfT_{\fM}\to \ov{\bfQ}_p$ (the reason for the subscript $\mF$ is the fact that these homomorphisms arise from newforms $\mF$ that are congruent to a certain Eisenstein series).  Write $\OL$ for the ring of integers of $L$. Suppose there exists $\varphi: \OL\to \ov{\bfQ}_p$, which separates the different $\lambda_{\mF}$s, i.e., such that $\varphi\circ \lambda_{\mF}=\varphi\circ \lambda_{\mF'}$ only if $\lambda_{\mF}=\lambda_{\mF'}$. 
[This is needed for proving that $\bfT_{1,\fm}$ is reduced  in  Proposition \ref{inject3} and showing $R \twoheadrightarrow \bfT^\Sigma$ in Proposition \ref{RtoT}.]

\begin{thm} \label{mainthm}
Assume $\dim_{\bfF}C_F^{\chi^{-1}}\otimes_{\Oo}\bfF=1$  and $C_F^{\chi}=0$.
 Assume further that at least one of the following conditions is satisfied: \begin{itemize}
\item[(i)] $e<p-1$ where $e$ is the ramification index of $p$ in $F$ or
\item[(ii)] $\chi=\omega^s$ for some integer $s$ or \item[(iii)] $\chi_N(p)\neq 1$.\end{itemize}  Then the $\Oo$-algebra map $\Phi: R \to \bfT^\Sigma$ given by $\tr \rho^{\rm univ}(\Frob_{\ell}) \mapsto (\varphi(c(\ell, \mF)) )_{\mF \in \mS} $ for all $\ell \not\in \Sigma$  is an isomorphism. Here $R$ is the ordinary universal deformation ring of $\rho_0$ defined in section \ref{deformationproblem}.
\end{thm}

The assumption that  $\dim_{\bfF}C_F^{\chi^{-1}}\otimes_{\Oo}\bfF=1$ is needed for Proposition \ref{uniqueness} and via this also Proposition \ref{infi}. Finally, the assumption that $C_F^{\chi}=0$ and assumptions (i)-(iii) are used in Theorem \ref{principality}.

\begin{proof}[Proof of Theorem \ref{mainthm}] The existence of the map $\Phi$ was proved in Proposition \ref{RtoT}. We apply Theorem 6.9 in \cite{BergerKlosin11} to the commutative diagram
$$\xymatrix{R \ar[r]^{\Phi}\ar[d]& \bfT^{\Sigma}\ar[d]\\R/I \ar[r]& \bfT^{\Sigma}/J^{\Sigma}}$$ noting that the top arrow is surjective and the bottom arrow is an isomorphism by Proposition \ref{RtoT} and $I$ is principal by Theorem \ref{principality}. 
 \end{proof}

Let us record a consequence of Theorem \ref{mainthm}. 
\begin{cor} \label{Jprincipal} The Eisenstein ideal $J^{\Sigma}$ is principal. 
\end{cor}

\subsection{Examples} \label{Examples}

In this section we will demonstrate  that our non-crossing assumption for the Hida families in  Proposition \ref{inject3} (and therefore also in Theorem \ref{mainthm}) is often satisfied. We will also present an example of a character $\chi$ that is unramified at $p$ and satisfies all the other assumptions (except that we were not able to check the non-crossing assumption). Such unramified characters cannot be handled by the methods of \cite{SkinnerWiles97}. 

As mentioned in the Introduction the related question of the geometry of the eigencurve at classical weight one points has been studied extensively. In particular, Bella{\"\i}che-Dimitrov \cite{BellaicheDimitrov}  prove that the eigencurve is smooth at such points if they are \emph{regular}, i.e. have distinct roots of the Hecke polynomial at $p$. This translates to the form lying in a unique Hida family up to Galois conjugacy. Our $p$-distinguishedness assumption ensures that our forms are regular. Bella{\"\i}che-Dimitrov further prove that the eigencurve is \'etale at such points if there does not exist a real quadratic field $K$ in which $p$ splits and such that the corresponding Galois representation becomes reducible over $K$. Note that our assumption that $\ov{\rho}^{\rm ss}=1 \oplus \chi$ rules out such a real multiplication case, as $\chi$ has to be odd. Our condition of not having Hida families (even Galois conjugate ones) cross at non-classical weight 1 specialisations corresponds to \'etaleness at all weight 1 specialisations of the connected components of the eigenvariety containing the Hida families $\mF \in \mS$. 

As far as we know, the geometry at non-classical points has not been studied. 
We can nevertheless exhibit many cases in which there is a unique Hida family, so that our non-crossing assumption is satisfied.
 In particular, \cite{BellaichePollack} calculate that for many irregular pairs $(p,k)$ (irregular in the sense that $p$ divides the Bernoulli number $B_{k}$) one indeed has $$\dim_\Lambda \bfT_\mathfrak{M}=1,\quad \textup{(unique Hida family)} $$ where $\bfT$ is the universal ordinary Hecke algebra of tame level $N=1$ and $\mathfrak{M}$ is a maximal ideal containing the Eisenstein ideal $J$ corresponding to the $\Lambda$-adic Eisenstein series which specializes at the particular weight $k$ to $$E_{k}=-\frac{B_{k}}{2}+ \sum_{n \geq 1} \sigma_{k-1}(n) q^n.$$
This corresponds to the Hecke algebra we considered for $\chi=\omega^{k-1}$.
Furthermore, by  Corollary 5.15 in \cite{Washingtonbook} we know that $p$-divisibility of $B_k$ implies that of $B_1(\omega^{k-1})=-L(0, \omega^{k-1})$. 

Set $N=1$ and fix a finite set $\Sigma$ satisfying assumptions (2) and (3). Since $p \mid B_k$ there exists a representation $\rho_0: G_\Sigma \to \GL_2(\bfF)$ of the form $$\rho_0=\begin{pmatrix}1 & *\\0& \ov{\chi} \end{pmatrix},$$ which does not split, but splits when restricted to $D_p$ (see section \ref{summary}). In this case (as $\chi$ is a power of $\omega$), the existence of $\rho_0$ can also be deduced from  Theorem 1.3 in \cite{Ribet76}.

We discuss the case $(p,k)=(37, 32)$ in detail to demonstrate that all our assumptions are satisfied for $\chi=\omega^{k-1}$. Indeed, since the class number of $F=\bfQ(\chi)=\bfQ(\zeta_{37})$ is 37, we know that $p \|B_1(\chi)$, $C_F^{\chi^{-1}}$ is a cyclic $\Oo$-module and $C_F^\chi=0$. We also have $\ov{\chi}|_{I_p} \neq 1$ as $\chi$ has order 36 and is ramified at $p$.

Theorem \ref{mainthm} therefore shows that $R=\bfT^\Sigma=\Oo$, where $R$ is the universal deformation ring of $\rho_0$ defined in section \ref{deformationproblem} and $\bfT^\Sigma$ is as in section \ref{summary}.  
This implies that there is a unique characteristic zero ordinary deformation of $\rho_0$  with determinant $\chi$, and this corresponds to a non-classical $p$-adic cuspform of weight 1, as $\chi$ is not quadratic (see Remark \ref{CM1}).

An example of a character $\chi$ unramified at $p$ that satisfies our assumptions is the following: there is an odd order 4 character of conductor 157, which is identified by its Conrey number of 28 (see \cite{lmfdb:157.d}). Using Sage \cite{sagemath} one can check that $L(0, \chi)$  is divisible by a prime above $5$ in $\bfQ(i)$, whereas  $L(0, \chi^{-1})$ is a 5-unit. This example (and others) can be found by using \cite{lmfdb} to search for  totally imaginary cyclic extensions with class number 5.

We could not check in this example whether the non-crossing assumption in Theorem \ref{mainthm} is satisfied as the coefficient field of the specialisation in weight 5 has degree 102 over $\bfQ$, so we could not confirm whether there is a unique Galois conjugate of this cuspform of weight 5 and level 157 congruent to $1+\chi$ for a fixed prime above 5.

\section{$R=T$ theorem in the split case} \label{sect8}
In this section we will treat the split case of the deformation problem for odd quadratic characters. We keep all the assumptions of section \ref{The residual representation}.

Let $\chi=\chi_{F/\bfQ}:\Gal(F/\bfQ)\to \bfZ_p^{\times}$ be the quadratic character associated to an imaginary quadratic extension $F/\bfQ$ (so $N={\rm cond}(\chi)=d_F$). We assume $p>2$ is inert in $F/\bfQ$ (to have $\ov{\chi}|_{D_p} \neq 1$). We note that assumption (1) in the case of a quadratic character implies that $C_F^\chi$ is a cyclic $\Oo$-module. This is actually equivalent to assuming that $C_F$ is a cyclic $\Oo$-module, as $\Gal(F/\bfQ)$ acts  on ${\rm Cl}(F)$ via $\chi$.

As before we will write $\tilde{\chi}:(\bfZ/d_F\bfZ)^{\times}\to \bfC^{\times}$ for the Dirichlet character associated with $\chi$. We will also denote by $\ov{\chi}$ the mod $\varpi$ reduction of $\chi$. 
The Dirichlet Class Number Formula and the functional equation imply that $L(0, \tilde{\chi})=2\frac{h_F}{w_F}$,  where $h_F$ is the class number of $F$ and $w_F:=\# \Oo_F^\times$.
If $p \mid h_F$ then there exists a non-split representation $\rho_0:G_\Sigma \to \GL_2(\bfF)$ of the form
 $$\rho_0=\bmat 1&*\\0&\ov{\chi} \emat,$$ which is split on $I_{\ell}$ for all primes, and also split on $D_p$ since $p \Oo_F$ is a principal ideal and therefore splits completely in the Hilbert class field. By Proposition \ref{uniqueness} this representation is unique up to isomorphism.

Write $S_1(d_F, \tilde{\chi})^{\rm CM}$ for the space of weight 1 classical cusp forms of level $d_F$ and character $\tilde{\chi}$ spanned by the set $\mN'$ of newforms with complex multiplication, i.e. such that for $f \in \mN'$ one has $a_\ell(f)\tilde{\chi}(\ell)=a_\ell(f)$ for all primes $\ell$, where $a_{\ell}(f)$ denotes the $T_{\ell}$-eigenvalue corresponding to $f$.  
Suppose that all the forms $f\in \mN'$ are defined over the extension $E/\bfQ_p$.

We define $\bfT^{\rm class}_1$ as the $\Oo-$subalgebra of $\prod_{f \in \mN'}\Oo$ generated by $(a_{\ell}(f))_{f}$ for all primes $\ell \not\in \Sigma$. 
By \cite{Serre77b} section 7.3 (see also \cite{DummiganSpencer} Proposition 2.4) each $f \in \mN'$ has the form $f=f_\varphi$, where $f_\varphi$ is induced from  a non-trivial non-quadratic character $\varphi: {\rm Cl}(F) \cong {\rm Gal}(H/F) \to \bfC^{\times}$ of finite order (for $H$ the Hilbert class field of $F$), with associated Galois representation $$\rho_{f_\varphi}={\rm ind}^\bfQ_F(\varphi)$$ and $\det \rho_{f_\varphi}=\chi$.
We recall that $a_\ell(f_\varphi)=0$ if $\ell$ is inert in $F/\bfQ$, and  
\be \label{CM} a_\ell(f_\varphi)=\varphi(\mathfrak{l})+\varphi(\mathfrak{l}^c) \text{ if } (\ell)= \mathfrak{l} \mathfrak{l}^c. \ee

We define the Eisenstein ideal $J \subset \bfT^{\rm class}_1$ as the ideal generated by $T_\ell -1 -\chi(\ell)$ for $\ell \notin \Sigma$ and $\mathfrak{m}$ the maximal ideal containing $J$.
We also note (see section 4.7 in \cite{Miyake89})  that the space of classical Eisenstein series of weight 1, level $d_F$ and character $\tilde{\chi}$ is spanned by the Eisenstein series $E_1(\tilde{\chi})$ with constant term $\frac{L(0,\tilde{\chi})}{2}$ at infinity and Hecke eigenvalues $1+\tilde{\chi}(\ell)$ for $ \ell \nmid d_F$.
Let $\mN$ be the subset of $\mN'$ of newforms $f$ congruent to $E_1(\tilde{\chi})$. We note that $(\bfT^{\rm class}_1)_\fm$ is naturally a subring of $\prod_{f \in \mN} \Oo$, which is of full rank as an $\Oo$-module.

Let $R^{\rm split}$ be the deformation ring defined in section \ref{deformationproblem}.
\begin{prop} \label{RsurjT}
We have an $\Oo$-algebra surjection $\Phi: R^{\rm split} \twoheadrightarrow (\bfT^{\rm class}_1)_\fm$ mapping ${\rm tr}(\rho^{\rm univ}(\Frob_\ell))$ to $T_\ell$ for all primes $\ell \not\in \Sigma$.
\end{prop}

\begin{proof}
We need to check that $\rho_{f_\varphi}$ with $\ov{\rho}_{f_\varphi}^{\rm ss} \cong 1 \oplus \ov{\chi}$ satisfies the deformation conditions. By the universality of $R^{\rm split}$ we then get surjections $R^{\rm split} \to \Oo$, which will induce $R^{\rm split} \twoheadrightarrow (\bfT^{\rm class}_1)_\fm$ as $R^{\rm split}$ is generated by traces (by  Proposition \ref{genbytraces}).

Condition (i) is clear.  For $\ell \mid d_F$  one can check that deformation condition (iii) is satisfied by Mackey theory. Indeed, for $\rho_{f_\varphi}={\rm ind}^\bfQ_F(\varphi)$ this gives ${\rm ind}^\bfQ_F(\varphi)|_{I_\ell}= {\rm ind}^{I_\ell}_{I_\lambda}(\varphi|_{I_\lambda})$ for $(\ell)=\lambda^2$, and therefore ${\rm ind}^\bfQ_F(\varphi)|_{I_\ell}={\rm ind}^{I_\ell}_{I_\lambda}(1)=1 \oplus \chi$.

For (ii) we note that $a_p(f_\varphi)=0$ as $p$ is inert in $F/\bfQ$. 
This means that the Hecke polynomial of an eigenform $f_\varphi$ is $x^2+\tilde{\chi}(p)=(x-1)(x-\tilde{\chi}(p))=(x-1)(x+1)$, which implies that $\rho_{f_\varphi}|_{D_p}$ is split.

\end{proof}

\begin{cor}
The Eisenstein ideal $J \subset (\bfT^{\rm class}_1)_\fm$ is principal.
\end{cor}

\begin{proof}
As the reducibility ideal $I^{\rm split} \subset R^{\rm split}$  is the smallest ideal $I$ of $R^{\rm split}$ such that ${\rm tr}(\rho^{\rm split}) \mod{I}$ is the sum of two characters, this means that $I^{\rm split}$ is contained in the ideal $I_0$ generated by ${\rm tr}(\rho^{\rm split}(\Frob_\ell))-1-\chi(\Frob_{\ell})$ for $\ell \notin \Sigma$. It follows from the proof of Proposition \ref{bound on R/I} that ${\rm tr}(\rho^{\rm split}) =1 + \chi \mod{I^{\rm split}}$, so $I^{\rm split}=I_0$.

Under $\Phi: R^{\rm split} \twoheadrightarrow (\bfT^{\rm class}_1)_\fm$ the generators of $I^{\rm split}$ map to the generators of $J$, so $\Phi(I^{\rm split})=J$, and the principality of $J$ follows from that of $I^{\rm split}$ (cf. part (1) of Theorem \ref{principality}).
\end{proof}

\subsection{Proving Eisenstein congruences in weight 1}

\begin{thm} \label{TJCM}
We have $\#((\bfT^{\rm class}_1)_\fm/J) \geq \#C_F$.
\end{thm}

\begin{proof}
We put $\bfT:=(\bfT^{\rm class}_1)_\fm$. Note that all $f \in \mN$ are defined over the completion $E'$ of $\bfQ(\mu_{p^n})^+$ at the prime above $p$,  where $p^n \| h_F$. As both sides of the inequality in the statement increase by the same factor if we extend the field $E$ we can and will assume that $E=E'$. The ramification index $e$ of $E$ over $\bfQ_p$ is $\frac{1}{2} \phi(p^n)= \frac{1}{2} p^{n-1}(p-1)$.

From \cite{BergerKlosinKramer14} Proposition 4.3 and the principality of $J$ we deduce the following (noting the correction made in \cite{BergerKlosin19} Remark 5.13 about the missing factor of $[E:\bfQ_p]$).

\begin{prop}\label{BKK}
For every $\Oo$-algebra morphism $\lambda: \bfT \to \Oo$ write $m_{\lambda}$ for the largest integer such that $1+ \chi(\ell) \equiv \lambda(T_\ell)$ mod $\varpi^{m_{\lambda}}$ for all $\ell \notin \Sigma$. 
Then \be \label{41} \frac{[E:\bfQ_p]}{e}\cdot \sum_{\lambda} m_{\lambda} = \val_p(\#\bfT/J).\ee
\end{prop}

We will show that \be \label{totaldepth} \frac{1}{e}\cdot \sum_{\lambda} m_{\lambda} \geq {\rm val}_p(h_F), \ee which together with Proposition \ref{BKK} implies the theorem, as we will now explain:

Indeed these would give us $$\frac{1}{[E:\bfQ_p]}\val_p(\#\bfT/J)=\frac{1}{e}\cdot \sum m_{\lambda}\geq \val_p(h_F) = n.$$ Hence one gets from this:
$$\val_p(\#\bfT/J)\geq [E:\bfQ_p]n=\val_p(\#C_F),$$ as desired.

Thus it remains to prove \eqref{totaldepth}. Consider a character $\varphi: {\rm Cl}(F) \cong {\rm Gal}(H/F) \to \ov{\bfQ}_p^\times$ of exact order $p^m$ for $1 \leq m \leq n$. We note (as in the proof of \cite{DummiganSpencer} Theorem 2.7) that since the values of $\varphi$ are $p^m$-th roots of unity we have $\varphi(\fq) \equiv 1 \mod{\varpi_m}$ for $\varpi_m$ the prime in $\bfQ(\mu_{p^m})$ above $p$ and $\fq$ any ideal of $\OF$. Note that $\varpi^{p^{n-m}}$ is a uniformizer in the completion of $\bfQ(\mu_{p^m})^+$ at the prime ideal above $p$. 

We deduce that 
$(\varphi+\varphi^c)(\fq) \equiv 2=1+\chi(\fq) \mod{\varpi^{p^{n-m}}}$ for any prime $q$ of $\bfZ$ which splits in $\OF$ as $\fq\ov{\fq}$. By \eqref{CM} this tells us that $m_\lambda\geq p^{n-m}$ for $\lambda$ corresponding to $f_\varphi$.

It remains to count how many such cusp forms $f_\varphi$ congruent to $E_1(\tilde{\chi})$ we have. Since $\varphi$ and $\varphi^{-1}=\varphi^c$ induce to the same cusp form we need to count how many (unordered) pairs $\{ \varphi, \varphi^{-1}\}$  with $\varphi$ exact order $p^m$ exist for each $1 \leq m \leq n$. Since we assume that $C_F={\rm Cl}(F) \otimes_{\bfZ} \Oo$ is cyclic, the $p$-part of ${\rm Cl}(F)$ is a cyclic abelian group $G$.

The order $p^m$ characters lie in a unique subgroup  of the character group of $G$ (which is isomorphic to $G$) of order $p^m$, which has $\phi(p^m)$ generators. We therefore have $\frac{1}{2}\phi(p^m)$ pairs $\{ \varphi, \varphi^{-1}\}$ with $\varphi$ exact order $p^m$.

Hence $$\frac{1}{e}\cdot \sum_{\lambda} m_{\lambda} \geq \frac{1}{e} \sum_{m=1}^n \frac{1}{2} \phi(p^m) \cdot p^{n-m}=\frac{2}{\phi(p^n)} \sum_{m=1}^n \frac{1}{2} \phi(p^n)=n.$$
This gives \eqref{totaldepth} and thus concludes the proof of the theorem.
\end{proof}

\begin{rem}
This bound on the congruence module $T/J$ cannot be proved by the usual methods: For the method used e.g. in  \cite{BergerKlosin19} Proposition 5.2 one needs a modular form with constant term a $p$-unit. However, the Eisenstein part of  $M_1(d_K, \tilde{\chi})$ is spanned by $E_1(\tilde{\chi})$, which has $\frac{1}{2}L(0,\tilde \chi)$ as constant term. Deducing the bound from Wiles's result Theorem \ref{Wiles2} is also difficult, as we only know $\bfT \twoheadrightarrow \bfT_1 \twoheadrightarrow \bfT_1^{\rm class}$ and would need to establish classicality of the specialisation in weight 1. In addition, one would need to show (for the splitting of the associated Galois representation at $p$) that the specialisation is a $p$-stabilisation of a form of level $d_F$.

\end{rem}

We obtain the following $R=T$ theorem in the split case:
\begin{thm} \label{CMresult}
Consider $F/\bfQ$ an imaginary quadratic field and $p>2$ inert in $F/\bfQ$ dividing the class number of $F$. Assume that $C_F$ is a cyclic $\Oo$-module (and assumptions (2) and (3) in section \ref{The residual representation}).
Then the map $\Phi: R^{\rm split} \to (\bfT^{\rm class}_1)_\fm$ in Proposition \ref{RsurjT} is an isomorphism.
\end{thm}

\begin{proof} By Proposition \ref{RsurjT} and the fact that $\Phi(I^{\rm split}) \subset J$  we get the following commutative diagram 
$$\xymatrix{R^{\rm split} \ar[r]^{\Phi}\ar[d] & (\bfT_1^{\rm class})_{\fm}\ar[d]\\R^{\rm split}/I^{\rm split} \ar[r] & (\bfT_1^{\rm class})_{\fm}/J}$$ where the top map is surjective. By combining Proposition \ref{bound on R/I} and Theorem \ref{TJCM} we see that the bottom arrow is an isomorphism. We can then apply Theorem 6.9 in \cite{BergerKlosin11} to conclude that $\Phi$ is also an isomorphism noting that $I^{\rm split}$ is principal by part (1) of Proposition \ref{principality}.
\end{proof}

\begin{rem}
This result complements the work of Castella and Wang-Erickson \cite{CastellaWangErickson21} on Greenberg's conjecture for ordinary cuspidal eigenforms $f$, who prove in the residually irreducible case that $\rho_f$ is split at $p$ if and only if $f$ is CM.
\end{rem}

Our theorem implies the following equivalence:
\begin{cor} \label{lastcorollary}
Under the assumptions of section \ref{The residual representation} let $\rho: G_\Sigma \to \GL_2(\Oo)$ be an  ordinary deformation of $\rho_0$, and assume that $\chi$ is unramified at $p$. Then $\rho$ is modular by a classical weight 1 form if and only if $\rho$ is unramified at $p$  and $\chi$ is quadratic.
\end{cor}

\begin{proof}
First note that $\rho$ unramified at $p$  is equivalent to $\rho|_{D_p}$ being split under our assumptions. Indeed,
if $\rho$ is unramified then $p$-distinguishedness forces $\ov{\rho}(\Frob_p)$ to have distinct eigenvalues and hence $\rho|_{D_p}$ is split.

We now assume that $\rho|_{D_p}$ is split  and $\chi$ is quadratic. As we assume that $\chi$ is unramified at $p$ and $\chi|_{D_p} \neq 1$ we deduce that $p$ is inert in the imaginary quadratic extension, which is the splitting field of $\chi$. By Theorem \ref{CMresult} we conclude that $\rho$ is modular by a classical weight 1 form.

Conversely, if $\rho$ is modular by a classical weight 1 form then by Remark \ref{rem5.5} we know that $\chi$ is quadratic and that $\rho(G_\Sigma)$ is finite hence, in particular, $\rho$ is split on $D_p$.
\end{proof}

\bibliographystyle{amsalpha}
\bibliography{standard2}

\end{document}